\documentclass[12pt,english]{article}
\usepackage[margin=1.5in]{geometry}

\usepackage{mdwlist}
\usepackage{enumerate}
\usepackage{amssymb,amsbsy,latexsym}
\usepackage{amsmath}
\usepackage{graphics, subfigure, float}
\usepackage{fp, calc}
\usepackage{hyperref}
\usepackage{url}

\usepackage[T1]{fontenc} 
\usepackage{fourier}
\usepackage{bm}

\usepackage{amscd,amsthm}

\usepackage{pst-all}
\usepackage{pstricks-add}
\usepackage{pst-func}
\newpsobject{showgrid}{psgrid}{subgriddiv=1,griddots=10,gridlabels=6pt}
\usepackage{verbatim, comment}
\usepackage{datetime}

\newtheoremstyle{theorem}{1em}{1em}{\slshape}{0pt}{\bfseries}{.}{ }{}
\theoremstyle{theorem}
\newtheorem{theorem}{Theorem}

\newtheorem{conjecture}{Conjecture}
\newtheorem*{theorem*}{Theorem}
\newtheorem{corollary}[theorem]{Corollary}
\newtheorem*{corollary*}{Corollary}
\newtheorem{proposition}[theorem]{Proposition}
\newtheorem*{proposition*}{Proposition}
\newtheorem{lemma}[theorem]{Lemma}
\newtheorem{definition}[theorem]{Definition}

\newtheorem*{claim*}{Claim}
\theoremstyle{remark}

\newtheorem*{remark*}{Remark}

\providecommand{\setN}{\mathbb{N}}
\providecommand{\setZ}{\mathbb{Z}}
\providecommand{\setQ}{\mathbb{Q}}
\providecommand{\setR}{\mathbb{R}}
\providecommand{\setC}{\mathbb{C}}
\newcommand{\E}{\mathop{\mathbb{E}}}

\newcommand{\Vol}{\mathrm{Vol}}
\newcommand{\rank}{\textrm{rank}}
\renewcommand{\span}{\textrm{span}}

\newcommand{\nd}{\mathrm{nd}}

\def\pazocal#1{#1}

\usepackage[displaymath,textmath,graphics, subfigure, floats]{preview} 
\PreviewEnvironment{center} 
\PreviewEnvironment{pspicture} 

\makeatother

\title{The Subspace Flatness Conjecture and Faster Integer Programming}
\author{Victor Reis\thanks{University of Washington, Seattle and Microsoft Research. Email: {\tt vrstcs@gmail.com}.} \;\; and \; Thomas Rothvoss\thanks{University of Washington, Seattle. Email: {\tt rothvoss@uw.edu}. Supported by NSF CAREER grant 1651861, a David \& Lucile Packard Foundation Fellowship and NSF grant 2318620 \emph{AF: SMALL: The Geometry of Integer Programming and Lattices}.}}
\date{}

\begin{document}

\maketitle

\begin{abstract} 
  In a seminal paper, Kannan and Lov\'asz (1988) considered a quantity $\mu_{KL}(\Lambda,K)$
  which denotes the best volume-based lower bound on the \emph{covering radius} $\mu(\Lambda,K)$ of a convex
  body $K$ with respect to a lattice $\Lambda$. Kannan and Lov\'asz proved that $\mu(\Lambda,K) \leq n \cdot \mu_{KL}(\Lambda,K)$ and the Subspace Flatness Conjecture by Dadush (2012) claims an $O(\log(2n))$ factor suffices, which would match
  the lower bound from the work of Kannan and Lov\'asz.
  We settle this conjecture up to a constant in the exponent by proving that  $\mu(\Lambda,K) \leq O(\log^{3}(2n)) \cdot \mu_{KL} (\Lambda,K)$. Our proof is
  based on the Reverse Minkowski Theorem due to Regev and Stephens-Davidowitz (2017).
  Following the work of Dadush (2012, 2019), we obtain a $(\log(2n))^{O(n)}$-time randomized algorithm to
  solve integer programs in $n$ variables.
  Another implication of our main result is a near-optimal \emph{flatness constant} of $O(n \log^{2}(2n))$,
  improving on the previous bound of  $O(n^{4/3} \log^{O(1)} (2n))$.
\end{abstract}


\section{Introduction}

Lattices are fundamental objects studied in various areas of mathematics and computer science.
Here, a \emph{lattice} $\Lambda$ is a discrete subgroup of $\setR^n$. If $B \in \setR^{n \times k}$ is a matrix with linearly independent columns
$b_1,\dots,b_k$, then we may write a lattice in the form $\Lambda(B) := \{ \sum_{i=1}^k y_ib_i: y_i \in \setZ\}$. 
In mathematics, lattices are the central object of study in the geometry of numbers with many applications
for example to number theory, see e.g. the textbooks of Gruber and Lekkerkerker~\cite{GeometryOfNumbers-GruberLekkerkerker1987} and Coppel~\cite[Chapter 8]{Numbertheory-Coppel-2009}. 
On the computer science side, lattices found applications for example in lattice-based cryptography~\cite{LWE-RegevJACM09} and cryptanalysis~\cite{KnapsackCryptosystems-Odlyzko1990}. One of the most important
algorithms in this area is the \emph{LLL-algorithm} by Lenstra, Lenstra and Lov\'asz~\cite{LLL1982} which finds an approximately orthogonal basis for a given lattice in polynomial time. One of the consequences
of the LLL-reduction is a polynomial time $2^{n/2}$-approximation algorithm for the problem of finding a (nonzero) \emph{shortest vector} in a lattice w.r.t. the Euclidean norm. Ajtai, Kumar, and Sivakumar~\cite{SVP-Sieving-Algo-AKS-STOC2001} invented the so-called \emph{sieving algorithm} to find the shortest vector in a lattice in time $2^{O(n)}$.
Later, Micciancio and Voulgaris~\cite{CVP-Voronoi-Algo-MicciancioVoulgaris-SICOMP2013} proved that an algorithm making use of the properties of the \emph{Voronoi cell} can even solve the more general \emph{closest vector problem} in time $2^{O(n)}$.
The sieving algorithm of \cite{SVP-Sieving-Algo-AKS-STOC2001} generalizes to arbitrary norms and beyond~\cite{BlomerNaeweGeneralizedSievingTCS2009,EnumerateLatticeAlgosDadushPeikertVempalaFOCS11,RandomizedSieving-Dadush2014}
while no such generalization is known for the Voronoi cell algorithm.
A more general problem with tremendous applications in combinatorial optimization and operations research is the one of finding an integer point in an arbitrary convex body or polytope.
Lenstra~\cite{IPinFixedDim-Lenstra1983} used the then-recent lattice basis reduction algorithm to solve any $n$-variable integer program in time\footnote{To be precise, the paper states a running time of $2^{O(n^3)}$ but a more careful look at the LLL-reduced basis that is being used immediately reduces it to $2^{O(n^2)}$.} $2^{O(n^2)}$. This was later improved by Kannan~\cite{n-to-n-algos-for-SVP-CVP-Kannan-MOR1987} to $2^{O(n)}n^{2.5n}$ and then by Hildebrand and K\"oppe~\cite{DBLP:journals/disopt/HildebrandK13} to $2^{O(n)}n^{2n}$ followed by Dadush, Peikert and Vempala~\cite{EnumerateLatticeAlgosDadushPeikertVempalaFOCS11} with $n^{\frac{4}{3}n + o(n)}$  and finally by Dadush~\cite{DadushThesis2012} to $2^{O(n)}n^n$.
Later Dadush, Eisenbrand and Rothvoss~\cite{FromApxToExactIP_DER_Arxiv2022} found a different, simpler algorithm providing the same $2^{O(n)}n^n$ time bound.

A parameter appearing in the geometry of numbers is the \emph{covering radius}  
\[
  \mu(\Lambda,K) := \min\big\{ r \geq 0 \mid \Lambda + rK = \textrm{span}(\Lambda) \big\}
\]
of a lattice $\Lambda \subseteq \setR^n$ with respect to a compact convex set $K \subseteq \setR^n$ with $\textrm{span}(\Lambda) = \textrm{affine.hull}(K)$. This quantity is substantially harder to compute and sits on the second level of the polynomial hierarchy. To be exact, given a lattice $\Lambda$ and a parameter $r>0$, it is $\Pi_2^P$-complete to decide if $\mu(\Lambda,B_{\infty}^n) \leq r$, see Haviv and Regev~\cite{HardnessofCoveringRadiusHavivRegev2006}.

In terms of approximating $\mu(\Lambda,K)$, one can quickly observe the lower bound $\mu(\Lambda,K) \geq (\frac{\det(\Lambda)}{\Vol_n(K)})^{1/n}$, simply  because for $r<(\frac{\det(\Lambda)}{\Vol_n(K)})^{1/n}$, the average density of the translates $\Lambda + rK$ is less than 1.
However, this lower bound may be far off the real covering radius, for example if $\Lambda = \setZ^2$
and $K = [-\frac{1}{M},\frac{1}{M}] \times [-M,M]$ with $M \geq 1$ in which case $\mu(\Lambda,K)=\frac{M}{2}$
while $(\frac{\det(\Lambda)}{\Vol_2(K)})^{1/2} = (\frac{1}{4})^{1/2} = \frac{1}{2}$.
On the other hand, for any subspace $W \subseteq \setR^n$ one trivially has $\mu(\Lambda,K) \geq \mu(\Pi_W(\Lambda),\Pi_W(K))$,
where $\Pi_W$ is the orthogonal projection onto $W$. Thus if we 
 abbreviate
\begin{equation} \label{eq:DefRW}
  R_W(\Lambda,K) := \Big(\frac{\det(\Pi_W(\Lambda))}{\Vol_{\dim(W)}(\Pi_W(K))} \Big)^{1/\dim(W)}
\end{equation}
then $\mu(\Lambda,K) \geq R_W(\Lambda,K)$ for any subspace\footnote{There is the technicality that for an arbitrary subspace $W$, the projection $\Pi_W(\Lambda)$ might not actually be a lattice because $\Pi_W(\Lambda)$ contains points that are arbitrarily close, e.g. when $\Lambda := \setZ^n$ for $n \geq 2$ and $W := \textrm{span}( e_1+\sqrt{2} e_2)$. So one should restrict to subspaces $W$ so that $\Pi_W(\Lambda)$ is indeed a lattice, which is equivalent to asking that $W^{\perp}$ is a lattice subspace w.r.t. $\Lambda$, which in turn is equivalent to asking that $W$ itself is a lattice subspace w.r.t. $\Lambda^*$; see Lemma 1.1 in \cite{KannanLovasz-CoveringMinima-AnnalsOfMath1988} for details. Alternatively, one can make the convention that $\det(\Pi_W(\Lambda)) = 0$ whenever $\Pi_W(\Lambda)$ fails to be a lattice, in which case the lower bound of $R_W(\Lambda,K)=0$ is true anyway.} $W$.
Following Kannan and Lov\'asz~\cite{KannanLovasz-CoveringMinima-AnnalsOfMath1988}, we consider a lower
bound $\mu_{KL}(\Lambda,K)$ on the covering radius which is defined as  
\[
 \mu_{KL}(\Lambda,K) := \max_{\substack{W \subseteq \textrm{span}(\Lambda)\textrm{ subspace} \\ d := \dim(W)}} \Big(\frac{\det(\Pi_W(\Lambda))}{\Vol_d(\Pi_W(K))}\Big)^{1/d} = \max_{\substack{W \subseteq \textrm{span}(\Lambda) \\ \textrm{subspace}}} \big\{ R_W(\Lambda,K) \big\}.
\]
Back to the example of  $\Lambda = \setZ^2$ and $K = [-\frac{1}{M},\frac{1}{M}] \times [-M,M]$ with $M \geq 1$, one can verify that an optimum choice is $W := \textrm{span}(e_1)$ so that $\det(\Pi_W(\Lambda)) = \det(\setZ)=1$ and $\Vol_1(\Pi_W(K)) = \Vol_1([-\frac{1}{M},\frac{1}{M}]) = \frac{2}{M}$, leading to $\mu_{KL}(\Lambda,K)=\frac{M}{2}$.
In fact, Kannan and Lov\'asz~\cite{KannanLovasz-CoveringMinima-AnnalsOfMath1988} provide an upper bound of
\[
 \mu_{KL}(\Lambda,K) \leq \mu(\Lambda,K) \leq n \cdot \mu_{KL}(\Lambda,K)
\]
for all full rank lattices $\Lambda \subseteq \setR^n$ and convex bodies $K \subseteq \setR^n$.
On the other hand, they also construct a simplex $K \subseteq \setR^n$ for which $\mu(\setZ^n,K) = H_n \cdot \mu_{KL}(\setZ^n,K)$ where $H_n =  \sum_{i=1}^n \frac{1}{i} = \log(n) \pm O(1)$ is the $n$th harmonic number.
Dadush~\cite{DadushThesis2012} states the following conjecture, attributing it to Kannan and Lov\'asz~\cite{KannanLovasz-CoveringMinima-AnnalsOfMath1988}: 

\begin{conjecture}[Subspace Flatness Conjecture] \label{conj:KL}
  For any full rank lattice $\Lambda \subseteq \setR^n$ and any convex body $K \subseteq \setR^n$ one has
\[
 \mu_{KL}(\Lambda,K) \leq \mu(\Lambda,K) \leq O(\log(2n)) \cdot \mu_{KL}(\Lambda,K).
\]
\end{conjecture}
Dadush also discovered the tremendous implications of this conjecture to optimization and showed that it would imply a $O(\log(2n))^n$-time algorithm to solve $n$-variable integer programs, assuming the subspace attaining $\mu_{KL}(\Lambda,K)$ could also be found in the same time~\cite[Theorem 7.1.3]{DadushThesis2012}.
Later, Dadush and Regev~\cite{TowardsReverseMinkowskiDadushRegevFOCS16} conjectured a \emph{Reverse Minkowski-type Inequality}, which intuitively says that
any lattice without dense sublattices should contain only a few short vectors. Among other applications,
they proved that this conjecture would
imply Conjecture~\ref{conj:KL} (with some logarithmic loss) at least for the case that $K$ is an ellipsoid.
The conjecture of \cite{TowardsReverseMinkowskiDadushRegevFOCS16} was then resolved by Regev and Stephens-Davidowitz~\cite{Regev-SD-ReverseMinkowskiTheoremSTOC17,Regev-SD-ReverseMinkowskiTheoremAnnalsOfMath2024}
with a rather ingenious proof. More precisely, they prove the following: 
\begin{theorem}[Reverse Minkowski Theorem~\cite{Regev-SD-ReverseMinkowskiTheoremSTOC17,Regev-SD-ReverseMinkowskiTheoremAnnalsOfMath2024}] \label{thm:ReverseMinkowskiTheorem}
  Let $\Lambda \subseteq \setR^n$ be a lattice that satisfies $\det(\Lambda') \geq 1$ for all sublattices $\Lambda' \subseteq \Lambda$.
  Then for a large enough constant $C>0$ and $s = C\log(2n)$ one has  $
   \rho_{1/s}(\Lambda) \leq \frac{3}{2}
  $.
\end{theorem}
Here, one has $\rho_t(x) := \exp(-\pi \|x/t\|_2^2)$ where $t>0$ and for a discrete set $S \subseteq \setR^n$ we abbreviate $\rho_t(S) := \sum_{x \in S} \rho_t(x)$.
Note that a set $\Lambda'$ is a \emph{sublattice} of a lattice $\Lambda$ if $\{ \bm{0}\} \subseteq \Lambda' \subseteq \Lambda$ and $\Lambda'$ is again a lattice.
To understand the power of this result compared to classical arguments,
note that from $\det(\Lambda') \geq 1$ for all $\Lambda' \subseteq \Lambda$ one can derive that each vector $x \in \Lambda \setminus \{ \bm{0}\}$ has length $\|x\|_2 \geq 1$ and so by a standard packing argument
we know that for any $r \geq 1$ one has $|\Lambda \cap rB_2^n| \leq (3r)^n$, which is exponential in $n$. On the other hand, again under the assumption that $\det(\Lambda') \geq 1$ for all $\Lambda' \subseteq \Lambda$, the Reverse Minkowski Theorem implies that
$|\Lambda \cap rB_2^n| \leq \exp( \Theta(\log^2(2n)) \cdot r^2)$ which is quasi-polynomial in $n$.
Also, \cite{Regev-SD-ReverseMinkowskiTheoremSTOC17,Regev-SD-ReverseMinkowskiTheoremAnnalsOfMath2024} tighten the reduction to the Subspace Flatness Conjecture and
show the following:
\begin{theorem}[\cite{Regev-SD-ReverseMinkowskiTheoremSTOC17,Regev-SD-ReverseMinkowskiTheoremAnnalsOfMath2024}] \label{thm:KLForEllipsoids} 
  For any full rank lattice $\Lambda \subseteq \setR^n$ and any ellipsoid $\pazocal{E} \subseteq \setR^n$ one has
  \[
   \mu_{KL}(\Lambda,\pazocal{E}) \leq \mu(\Lambda,\pazocal{E}) \leq O(\log^{3/2} (2n)) \cdot \mu_{KL}(\Lambda,\pazocal{E}).
  \]
\end{theorem}

\subsection{Our contribution}

Our main result is as follows:
\begin{theorem} \label{thm:KLConj} 
  For any full rank lattice $\Lambda \subseteq \setR^n$ and any convex body $K \subseteq \setR^n$
  one has
  \[
   \mu_{KL}(\Lambda,K) \leq \mu(\Lambda,K) \leq O(\log^3 (2n)) \cdot \mu_{KL}(\Lambda,K).
  \]
\end{theorem}
We will break the proof into two parts that can be found in Section~\ref{sec:WholeMainProof}.

Our result is constructive in the following sense: 
\begin{theorem} \label{thm:FindingSubspaceIn2ToN} 
  Given a full rank lattice $\Lambda := \Lambda(B)$ and a convex body $K \subseteq \setR^n$ with $c+r_0B_2^n \subseteq K \subseteq r_1 B_2^n$, there is a randomized
  algorithm to find a subspace $W \subseteq \setR^n$ so that 
  \[
    \mu(\Lambda,K) \leq O(\log^4 (2n)) \cdot R_W(\Lambda,K). 
  \]
  The running time of that algorithm is $2^{O(n)}$ times a polynomial in $\log(\frac{1}{r_0})$, $\log(r_1)$ and in the encoding length of $B$.
\end{theorem}
Here, a separation oracle suffices for $K$. See Section~\ref{sec:FindingSubspace} for a proof.
Following the framework laid out by Dadush~\cite{DadushThesis2012}, this implies a faster algorithm
to find a lattice point in a convex body:
\begin{theorem} \label{thm:SolvingIPinLogNtoN}
  Given a convex body $K \subseteq r B_2^n$ represented by a separation oracle and a lattice $\Lambda = \Lambda(B)$, there is a randomized
  algorithm that with high probability finds a point in $K \cap \Lambda$ or  correctly decides that there is none.
  The running time is  $(\log(2n))^{O(n)}$ times a polynomial in $\log(r)$ and the encoding length of $B$.
\end{theorem}
The proof can be found in Section~\ref{sec:IP}. 
Applying Theorem~\ref{thm:SolvingIPinLogNtoN} to integer programming we obtain the following:

\begin{corollary} \label{cor:SolvingExplicitIPinLogNtoN}
  Given $A \in \setQ^{m \times n}$, $b \in \setQ^m$ and $c \in \setQ^n$, the integer linear program
  $
  \max\{ c^Tx \mid Ax \leq b, x \in \setZ^n\}
  $
  can be solved in time $(\log(2n))^{O(n)}$ times a polynomial in the encoding length of $A$, $b$ and $c$.
\end{corollary}
An immediate consequence of our main result (Theorem~\ref{thm:KLConj}) is that $K$ can be replaced by a larger \emph{symmetric} body without decreasing the covering radius significantly:
\begin{theorem} \label{thm:CoveringRadiusKvsKminusK} 
  For any full rank lattice $\Lambda \subseteq \setR^n$ and any convex body $K \subseteq \setR^n$ one has
  \[\mu(\Lambda,K-K) \leq \mu(\Lambda,K) \leq O(\log^3(2n)) \cdot \mu(\Lambda,K-K).\]
\end{theorem}
Here, $K - K = \{ x-y \mid x,y \in K\}$ is the \emph{difference body} of $K$.
In yet another paper, Dadush~\cite{Dadush-Finding-DenseLatticeSubspacesSTOC19} had already studied the quantity $C_{\textrm{sym}}(n)$ which is the
supremum of the ratio $\frac{\mu(\Lambda,K)}{\mu(\Lambda,K-K)}$ over all $n$-dimensional $\Lambda$ and $K$. In particular if we let $C_{KL}(n)$ denote
the supremum of $\frac{\mu(\Lambda,K)}{\mu_{KL}(\Lambda,K)}$ over all $n$-dimensional $\Lambda$ and $K$, then Dadush~\cite{Dadush-Finding-DenseLatticeSubspacesSTOC19} pointed out that  $C_{\textrm{sym}}(n) \leq 4 \cdot C_{KL}(n)$.
In terms of a lower bound, the standard simplex $\mathrm{conv}\{\bm{0}, e_1, \dots, e_n\}$ yields $C_{\textrm{sym}}(n) \ge \frac{4n}{n+1}$, and it is unknown whether $C_{\textrm{sym}}(n) \to \infty$ as $n$ grows.

Another consequence of our main result is that the \emph{flatness constant} in dimension $n$ is bounded by $O(n\log^{2}(2n))$,
which is an improvement from the previously known bound of $O(n^{4/3} \log^{O(1)} (2n))$ obtained by combining the result of Rudelson~\cite{Rudelson1998DistancesBN} with \cite{BanaszczykLitvakPajorSzarekMOR99Flatness}. 
Here, for a lattice $\Lambda$ and a symmetric convex body $Q$, $\lambda_1(\Lambda,Q)$ denotes the length of a shortest non-zero vector in the lattice $\Lambda$ w.r.t. norm $\| \cdot \|_Q$. Moreover, $\Lambda^*$ is the dual lattice for $\Lambda$ and $K^{\circ}$ is the polar of $K$. 
\begin{theorem} \label{thm:FlatnessConstant}
  For any convex body $K \subseteq \setR^n$ and any full rank lattice $\Lambda \subseteq \setR^n$ one has
  \[
   \mu(\Lambda,K) \cdot \lambda_1(\Lambda^{*}, (K-K)^{\circ}) \leq O(n \log^{2}(2n)).
  \]
\end{theorem}
It is well known that Theorem~\ref{thm:FlatnessConstant} can also be rephrased in the following convenient form:
\begin{corollary} \label{cor:FlatnessConstantSimple}
  Let $K \subseteq \setR^n$ be a convex body with $K \cap \setZ^n = \emptyset$.
  Then there is a vector $c \in \setZ^n \setminus \{ \bm{0}\}$ so that at most $O(n \log^{2}(2n))$ many hyperplanes of the form $\langle c, x \rangle = \delta$ with $\delta \in \setZ$ intersect $K$.
\end{corollary}
We will prove Theorem~\ref{thm:CoveringRadiusKvsKminusK}, Theorem~\ref{thm:FlatnessConstant}
and Corollary~\ref{cor:FlatnessConstantSimple} in Section~\ref{sec:Implications}.

\section{Preliminaries}

In this section, we introduce the tools that we rely on later. We write $A \lesssim B$ if there is a universal constant
$C>0$ so that $A \leq C \cdot B$ holds. We write $A \asymp B$ if both $A \lesssim B$ and $B \lesssim A$ hold.

\subsection{Lattices}
A \emph{lattice} is any set of the form  $\Lambda = \Lambda(B) = \{ By \mid y \in \setZ^k\}$ where $B \in \setR^{n \times k}$ is a matrix with linearly independent columns. Here, the matrix $B$ is also called the \emph{basis}
of the lattice.
We define the \emph{rank} of the lattice as $\rank(\Lambda) := k = \dim(\textrm{span}(\Lambda))$ and a lattice $\Lambda \subseteq \setR^n$ with $\rank(\Lambda)=n$ has \emph{full rank}. Equivalently, one can define a lattice $\Lambda$ as a \emph{subgroup} of $(\setR^n,+)$ which is \emph{discrete}, i.e. there is a $\delta > 0$ (depending on $\Lambda$)
so that $\|x\|_2 \geq \delta$ for all $x \in \Lambda \setminus \{ \bm{0}\}$.
The \emph{fundamental parallelepiped} of a lattice $\Lambda := \Lambda(B)$ with $B \in \setR^{n \times k}$ is the convex set $\pazocal{P}(B) := \{ By \mid y \in [0,1)^k \}$. While the fundamental parallelepiped depends on the choice of the basis $B$, the volume  is the same for any basis of $\Lambda$ and it denotes an important invariant called the \emph{determinant} of the lattice, abbreviated by $\det(\Lambda) := \Vol_k(\pazocal{P}(B))$. In matrix notation one can
show that $\det(\Lambda) = \sqrt{\det(B^TB)}$; in the case that $\Lambda$ has full rank and hence $B \in \setR^{n \times n}$, this expression simplifies to $\det(\Lambda) = |\det(B)|$.
 We recommend the book of Matousek~\cite[Chapter 2]{LecturesOnDiscreteGeometryMatousek2002} for a readable and gentle introduction to lattices.

For a lattice $\Lambda \subseteq \setR^n$, we define the
\emph{dual lattice} as $\Lambda^* := \{ x \in \textrm{span}(\Lambda) \mid \left<x,y\right> \in \setZ \; \forall y \in \Lambda \}$. 
From this definition one can see that a lattice $\Lambda$ and its dual $\Lambda^*$ have the same rank and span the same subspace. For example if $\Lambda$ is a full rank lattice with basis $B \in \setR^{n\times n}$, then $(B^{-1})^T$
is a basis for $\Lambda^*$.
One can prove the following
convenient relation:
\begin{lemma} \label{lem:DetOfPolar}
  For any lattice $\Lambda \subseteq \setR^n$ one has $\det(\Lambda) \cdot \det(\Lambda^*) = 1$.
\end{lemma}

A set $K \subseteq \setR^n$ is called a \emph{convex body} if it is convex, compact (i.e. bounded and closed) and
has a non-empty \emph{interior} $\textrm{int}(K) := \{ x \in K \mid \exists \varepsilon > 0 : (x + \varepsilon B_2^n) \subseteq K\}$. In particular, a convex body is full dimensional.
A set $Q$ is called \emph{symmetric} if $-Q=Q$. In particular, a symmetric convex body $Q$ is the unit ball of a norm that we write as $\| \cdot \|_Q$.
In order to minimize confusion, throughout this manuscript we will always denote symmetric convex bodies
with $Q$ and (potentially) asymmetric convex bodies by $K$.
For a symmetric convex body $Q$, the norm $\|x\|_Q$ is defined as the least scaling $r \ge 0$ so that $x \in rQ$. For a lattice $\Lambda$ and a symmetric convex body $Q$ we denote
the \emph{length of the shortest vector} w.r.t. norm $\| \cdot \|_Q$ by
\[
  \lambda_1(\Lambda,Q) := \min_{x \in \Lambda \setminus \{ \bm{0}\}} \|x\|_Q.
\]
A very classical result by Minkowski states that lattices must contain short vectors:
\begin{theorem}[Minkowski's First Theorem (1889)] \label{thm:Minkowski}
Let $\Lambda \subseteq \setR^n$ be a full rank lattice and $Q \subseteq \setR^n$ be a symmetric convex body. Then $\lambda_1 (\Lambda, Q) \le 2 \Big(\frac{\det(\Lambda)}{\Vol_n(Q)}\Big)^{1/n}$.
\end{theorem} 
The famous LLL algorithm can find a lattice basis that has bounded \emph{orthogonality defect}. In particular the first column of the computed
basis matrix satisfies Minkowski's Theorem (Theorem~\ref{thm:Minkowski}) up to a $2^{O(n)}$ factor.  
\begin{theorem}[The LLL algorithm --- Lenstra, Lenstra, Lov\'asz~{\cite[Prop 1.6+Prop 1.11]{LLL1982}}] \label{thm:LLLalgorithm}
  Let $A \in \setQ^{n \times n}$ be the given basis of a lattice $\Lambda := \Lambda(A)$. In time polynomial in $n$
  and the encoding length of $A$ one can compute a basis $B = (b_1,\dots,b_n)$ of $\Lambda$ so that
  \begin{enumerate}
  \item[(i)] $\|b_1\|_2 \leq 2^{(n-1)/2} \cdot \lambda_1(\Lambda,B_2^n)$
  \item[(ii)] $\|b_1\|_2 \leq 2^{(n-1)/4} \cdot \det(\Lambda)^{1/n}$
  \item[(iii)] $\det(\Lambda) \leq \prod_{j=1}^n \|b_j\|_2 \leq 2^{n(n-1)/4} \det(\Lambda)$
  \end{enumerate}
\end{theorem}

As mentioned before, for any $s>0$ and $x \in \setR^n$ we denote the (unnormalized) Gaussian density by
\[
  \rho_s(x) := \exp(-\pi\|x/s\|_2^2).
\]
 For a function $f : \setR^n \to \setC$ with $\int_{\setR^n} |f(x)| \; dx < \infty$, we define its
\emph{Fourier transform} as the function $\hat{f} : \setR^n \to \setC$ with
\[
 \hat{f}(y) := \int_{\setR^n} f(x) \cdot e^{-2\pi i\left<x,y\right>} \; dx
\]
for $y \in \setR^n$. For a discrete set $S$ we write $f(S) := \sum_{x \in S} f(x)$.
\begin{theorem}[Poisson Summation Formula~{\cite[Chapter 8, 31.46.e]{HewittRossVolII1970}}] \label{thm:PoissonSummationFormula}
For any full rank lattice $\Lambda \subseteq \setR^n$ and any Schwartz function $f : \setR^n \to \setC$ one has  
\[
 f(\Lambda) = \det(\Lambda^*) \cdot \hat{f}(\Lambda^*)
\]
\end{theorem} 
While we avoid a formal definition, a \emph{Schwartz function} is any smooth function whose derivatives decay rapidly.
Applying Theorem~\ref{thm:PoissonSummationFormula} to the Schwartz function $f(x) := \rho_s(x + u)$
and making use of the convenient fact that the Fourier transform of $\rho_s$ is $s^n \cdot \rho_{1/s}$,
provides the following identity:
\begin{lemma}[Implication of Poisson Summation Formula I] \label{lem:ImplOfPoissonSummationFormulaI}
  For any full rank lattice $\Lambda \subseteq \setR^n$, any $s>0$ and $u \in \setR^n$ one has
  \[
    \rho_s(\Lambda + u) = s^n\det(\Lambda^*) \cdot \sum_{y \in \Lambda^*} \rho_{1/s}(y) \cdot e^{2\pi i\left<u,y\right>}
  \]  
\end{lemma}
For example, one consequence of Lemma~\ref{lem:ImplOfPoissonSummationFormulaI} is that $\rho_s(\Lambda + u) \leq \rho_s(\Lambda)$ for all $u \in \setR^n$. For us, another variant of Lemma~\ref{lem:ImplOfPoissonSummationFormulaI} will be relevant which is obtained by observing that $|e^{2\pi i\left<u,y\right>}| \leq 1$ for all $y \in \setR^n$,
while $e^{2\pi i\left<u,\bm{0}\right>} = 1$.
\begin{lemma}[Implication of Poisson Summation Formula II~{\cite[Lemma 1.1.(i)]{TransferenceTheorems-Banaszczyk93}}] \label{lem:GeneralLatticeShiftApx}
  For any full rank lattice $\Lambda \subseteq \setR^n$, vector $u \in \setR^n$ and any $s>0$ one has
  \[
   |\rho_s(\Lambda + u)-s^n\det(\Lambda^*)| \leq s^n\det(\Lambda^*) \cdot \rho_{1/s}(\Lambda^* \setminus \{\bm{0}\}).
  \]
\end{lemma}
This lemma should be understood as the deep fact that if the discrete Gaussian with parameter $\frac{1}{s}$
has little weight on non-zero dual lattice vectors, then the discrete Gaussian with parameter $s$ is
approximately invariant under shifts, i.e. the value $\rho_s(\Lambda + u)$ is approximately the same for all $u$.
For background on all these Fourier-analytic aspects, we recommend the excellent notes of Regev~\cite{RegevLectureNotes2009}.

\subsection{Stable lattices and the canonical filtration} \label{subsec:StableLattices}

As mentioned in the introduction, a set $\Lambda'$ is a \emph{sublattice} of a lattice $\Lambda$ if $\{ \bm{0}\} \subseteq \Lambda' \subseteq \Lambda$ and $\Lambda'$ is again a lattice.
A subspace $W \subseteq \setR^n$ is a \emph{lattice subspace} of a lattice $\Lambda \subseteq \setR^n$
if $\textrm{span}(W \cap \Lambda) = W$. Similarly, a sublattice $\Lambda' \subseteq \Lambda$ is called
\emph{primitive} if there is a subspace $W$ with $\Lambda \cap W = \Lambda'$. Phrased differently, a primitive sublattice
contains all lattice points in its span.
\begin{definition} \label{def:QuotientLattice}
For a lattice $\Lambda$ and a primitive sublattice $\Lambda' \subseteq \Lambda$, we define
the \emph{quotient lattice} as $\Lambda / \Lambda' := \Pi_{\textrm{span}(\Lambda')^{\perp}}(\Lambda)$.
\end{definition}
Here $\Pi_{\textrm{span}(\Lambda')^{\perp}}$ denotes the orthogonal
projection onto the orthogonal complement of $\textrm{span}(\Lambda')$.
In many ways one can imagine that the quotient operation
decomposes $\Lambda$ into two lattices $\Lambda'$ and $\Lambda / \Lambda'$.
In particular $\Lambda'$ and $\Lambda / \Lambda'$ span orthogonal subspaces. 
\begin{lemma}[Determinant of quotient lattice~{\cite[Lemma 1.8]{Grayson1984}}]
For any lattice $\Lambda \subseteq \setR^n$ and any primitive sublattice $\Lambda' \subseteq \Lambda$ one has $\det(\Lambda) = \det(\Lambda') \cdot \det(\Lambda / \Lambda')$. 
\end{lemma}
A property that will turn out to be important throughout this manuscript is the following:
\begin{definition}
A lattice $\Lambda \subseteq \setR^n$ is called \emph{stable} if $\det(\Lambda)=1$
and $\det(\Lambda') \geq 1$ for all sublattices $\Lambda' \subseteq \Lambda$.
\end{definition}
That means a stable lattice does not
contain any sublattice that is denser than the lattice itself. One can easily verify that
for example $\setZ^n$ is stable. A non-trivial fact is as follows:
\begin{lemma}[Stability of the dual lattice ({\cite[Discussion 7.11]{Grayson1984}} and {\cite[Proposition 2.2.2.(i)]{PhDThesisStephens-Davidowitz2017}}] \label{lem:StabilityOfDual}
  A lattice $\Lambda$ is stable if and only if the dual $\Lambda^*$ is stable.
\end{lemma}
We denote  $\textrm{nd}(\Lambda) := \det(\Lambda)^{1/\textrm{rank}(\Lambda)}$ as the \emph{normalized determinant} of a lattice $\Lambda$ with the convention that  $\textrm{nd}(\{\bm{0}\}) =1$.
Of course not every lattice is stable, but crucially every lattice can in some sense be decomposed into
lattices that are stable (after appropriate scaling). Fix any lattice $\Lambda \subseteq \setR^n$
and consider the 2-dimensional point set
\[
 S := \big\{ (\textrm{rank}(\Lambda'), \ln(\det(\Lambda'))) \mid \textrm{sublattice }\Lambda' \subseteq \Lambda \big\}
\]
which is also called the \emph{canonical plot}. The convex hull $\textrm{conv}(S)$ of this point set is called the \emph{canonical polygon}. As $\textrm{conv}(S)$ is unbounded in the \((0,1)\) direction, all extreme points of $\textrm{conv}(S)$ must lie on its lower envelope. These extreme points have fascinating properties:

\begin{proposition}[{\cite[Lemma 1.18]{Grayson1984}}] \label{prop:UniquenessInCanonicalPolygon}
  Let $\Lambda \subseteq \setR^n$ be a lattice and let $S \subseteq \setR^2$ be its canonical plot.
  Then every extreme point of $\textrm{conv}(S)$ corresponds to a unique sublattice of $\Lambda$.
\end{proposition}

 \iftrue
 \begin{center}
   \psset{unit=0.8cm}
  \begin{pspicture}(0,-1.8)(5,3.6)
     \pnode(0,0){A0}
    \pnode(1,-1){A1}
    \pnode(2,-1.5){A2}
    \pnode(4,-0.25){A3}
    \pnode(5,1){4}
    \pnode(6,3){L5}

    \cnode*(0,0){2.5pt}{L0}
    \cnode*(1,-1){2.5pt}{L1}
    \cnode*(2,-1.5){2.5pt}{L2}
    \cnode*(4,-0.25){2.5pt}{L3}
    \cnode*(6,2){2.5pt}{L5}
    \pspolygon[fillstyle=solid,fillcolor=lightgray,linestyle=none](0,2)(L0)(L1)(L2)(L3)(L5)
    \psline(0,2)(L0)(L1)(L2)(L3)(L5)
    \psaxes[ticks=none,labels=none,arrowsize=5pt]{->}(0,0)(0,-1)(6,3.5) \rput[c](6,8pt){$\textrm{rank}(\Lambda')$}
    \psdots[linewidth=1.5pt](L0)(L1)(L2)(L3)(L5)
    \rput[l](5pt,3.5){$\ln(\det(\Lambda'))$}
    \nput{180}{L0}{$\Lambda_0 = \{ \bm{0}\}$}
    \nput{-90}{L2}{$\Lambda_{i-1}$}
    \nput{-60}{L3}{$\Lambda_i$}
    \nput{0}{L5}{$\Lambda_k=\Lambda$}
    \rput[c](2,1){$\textrm{conv}(S)$}
  \end{pspicture}
\end{center}
\fi

Now we can associate the extreme points with their sublattices:
\begin{proposition}[Properties of the canonical filtration ({\cite[Lem 1.18+Discussion 1.27]{Grayson1984}} and {\cite[Prop 2.2.2]{PhDThesisStephens-Davidowitz2017}})] \label{prop:PropertiesOfCanonicalFiltration}
  Let $\Lambda \subseteq \setR^n$ be a lattice and let $S \subseteq \setR^2$ be its canonical plot.
  Let $\Lambda_0,\dots,\Lambda_k$ be the sublattices corresponding to the extreme points of the canonical polygon $\textrm{conv}(S)$, ordered so that $\textrm{rank}(\Lambda_0) < \textrm{rank}(\Lambda_1) < \dots < \textrm{rank}(\Lambda_k)$. Then
  \begin{enumerate}
  \item[(i)] The sublattices form a chain $\{ \bm{0}\} = \Lambda_0 \subset \Lambda_1 \subset \dots \subset \Lambda_k = \Lambda$. 
  \item[(ii)] The quotient lattices $\Lambda_i / \Lambda_{i-1}$ are all scalar multiples of stable lattices, i.e.
  $\frac{1}{r_i} \cdot (\Lambda_i / \Lambda_{i-1})$ is stable for all $i \in \{ 1,\dots,k\}$ where $r_i := \textrm{nd}(\Lambda_i / \Lambda_{i-1})$.
\item[(iii)] The normalized determinants are increasing, i.e. $r_1 < \dots < r_k$.
  \end{enumerate}
\end{proposition}
We want to give this chain a name:
\begin{definition}
  Let $\Lambda \subseteq \setR^n$ be a lattice. Then the chain $\{ \bm{0}\} = \Lambda_0 \subset \Lambda_1 \subset \dots \subset \Lambda_k = \Lambda$ of sublattices from Prop~\ref{prop:PropertiesOfCanonicalFiltration}
  is called the \emph{canonical filtration} of $\Lambda$.
\end{definition}
These concepts and results go back to papers by Harder and Narasimhan~\cite{Harder1974}, Stuhler~\cite{Stuhler1976} and
Grayson~\cite{Grayson1984}. Detailed proofs can also be found in the thesis of Stephens-Davidowitz~\cite{PhDThesisStephens-Davidowitz2017} and for an excellent historic account on the origins we recommend Casselman~\cite{ajm/1118669693}.

It will be useful to replace the canonical filtration by an \emph{approximate} filtration where
the normalized determinants grow exponentially. 
We make the following definition:
\begin{definition}
Let $t \geq 1$.  We call a lattice $\Lambda \subseteq \setR^n$ \emph{$t$-stable}  if
  the following holds:
  \begin{enumerate}
  \item[(I)] For any sublattice $\tilde{\Lambda}  \subseteq \Lambda$ one has  $\nd(\tilde{\Lambda}) \geq \frac{1}{t}$.
  \item[(II)] For any sublattice $\tilde{\Lambda}  \subseteq \Lambda^*$ one has   $\nd(\tilde{\Lambda}) \geq \frac{1}{t}$.
  \end{enumerate}
\end{definition}
Note that a lattice is 1-stable if and only if it is stable. It might be instructive to observe that
(I) implies that $\det(\Lambda) \geq \frac{1}{t^{\textrm{rank}(\Lambda)}}$ while (II) implies that $\det(\Lambda) \leq t^{\textrm{rank}(\Lambda)}$.
We can similarly define $t$-stable filtrations:
\begin{definition} \label{def:tStableFiltration}
  Given a lattice $\Lambda \subseteq \setR^n$, we call a sequence $\{ \bm{0}\} = \Lambda_0 \subset \dots \subset \Lambda_{k} = \Lambda$ a \emph{$t$-stable filtration} of $\Lambda$ if the following holds:
  \begin{enumerate}
  \item[(a)] The normalized determinants $r_i := \nd(\Lambda_i/\Lambda_{i-1})$ satisfy
    $r_1 < \dots < r_{k}$.
  \item[(b)] The lattices $\frac{1}{r_i} (\Lambda_i / \Lambda_{i-1})$ are $t$-stable for all $i=1,\dots,k$.
  \end{enumerate}
  We call a $t$-stable filtration \emph{well-separated} if additionally the following holds: 
  \begin{enumerate}
  \item[(c)] One has $r_i \leq \frac{1}{2} r_{i+2}$ for all $i=1,\dots,k-2$.
  \end{enumerate}
\end{definition}

For example, the canonical filtration is $1$-stable. It turns out we can make any $t$-stable filtration well-separated:
\begin{theorem}\label{thm:ApproxFilt}
  Given a lattice $\Lambda \subseteq \setR^n$ and a $t$-stable filtration $\{ \bm{0}\} = \Lambda_0 \subset \dots \subset \Lambda_{k} = \Lambda$, in polynomial time we can compute a $2t$-stable well-separated filtration $\{ \bm{0}\} = \tilde{\Lambda}_0 \subset \dots \subset \tilde{\Lambda}_{\tilde{k}} = \Lambda$.
\end{theorem}
We defer the proof to Section~\ref{appendix:ApproximateFiltration}. Using the canonical filtration
as input to Theorem~\ref{thm:ApproxFilt} yields:

 \begin{corollary} \label{cor:ExistenceTwoStableWellSeparatedFiltration}
   For any lattice $\Lambda \subseteq \setR^n$, there exists a 2-stable well-separated filtration $\{ \bm{0}\} = \Lambda_0 \subset \dots \subset \Lambda_{k} = \Lambda$.
\end{corollary}


We collect a few more properties of $t$-stable lattices. In particular, due to the Reverse Minkowski Theorem,
the discrete Gaussian weight with parameter $\Theta(t\log(n))$ of a $t$-stable lattice is approximately invariant under shifts.
\begin{lemma}[Properties of $t$-stable lattices] \label{lem:PropertiesOf2StableLattices} 
  There is a universal constant $C>0$ so that the following holds: 
  Let $\Lambda$ be a $t$-stable lattice for $t \geq 1$. Then for $s = C\log(2n)$ one has
  \begin{enumerate}
  \item[(a)] $\Lambda^*$ is $t$-stable.
  \item[(b)]   $\rho_{1/(st)}(\Lambda) \leq \frac{3}{2}$. 
  \item[(c)] For any $u \in \setR^n$ one has $\frac{\rho_{st}(\Lambda + u)}{\rho_{st}(\Lambda)} \geq \frac{1}{3}$.
  \end{enumerate}
\end{lemma}
\begin{proof}
  (a) is immediate from the definition of $t$-stability. Next, let $s = C\log(2n)$ be the parameter from
  Theorem~\ref{thm:ReverseMinkowskiTheorem}.
  For (b), 
  we can see that for any $\Lambda' \subseteq t\Lambda$ one has $\det(\Lambda') \geq 1$
  and so the Reverse Minkowski Theorem (Theorem~\ref{thm:ReverseMinkowskiTheorem}) applies to
  the lattice $t\Lambda$. Then  $\rho_{1/(st)}(\Lambda) = \rho_{1/s}(t\Lambda) \leq \frac{3}{2}$ which gives (b).
  For (c), applying Lemma~\ref{lem:GeneralLatticeShiftApx} twice gives
  \[
   \frac{\rho_{st}(\Lambda + u)}{\rho_{st}(\Lambda)} \geq \frac{(st)^n \det(\Lambda^*) \cdot (1-\rho_{1/(st)}(\Lambda^* \setminus \{\bm{0}\}))}{(st)^n \det(\Lambda^*) \cdot (1+\rho_{1/(st)}(\Lambda^* \setminus \{ \bm{0}\}))} \stackrel{(a)+(b)}{\geq} \frac{1-\frac{1}{2}}{1+\frac{1}{2}} = \frac{1}{3}. \qedhere
  \]
\end{proof}

\subsection{Convex geometry}\label{sec:LValueAndVolEstimates}

We review a few results from convex geometry that can all be found in the textbook
by Artstein-Avidan, Giannopoulos and Milman~\cite{AsymptoticGeometricAnalysis-Book2015}.
We denote $B_2^n := \{ x \in \setR^n \mid \|x\|_2 \leq 1\}$
and $S^{n-1} := \{ x \in \setR^n \mid \|x\|_2=1\}$ as the Euclidean ball and sphere, respectively.
The \emph{relative interior}
of $K$ is $\textrm{rel.int}(K) := \{ x \in K \mid \exists \varepsilon > 0: (x + \varepsilon \cdot B_2^n) \cap \textrm{affine.hull}(K) \subseteq K \}$.
We define the \emph{mean width} of a convex body $K$ as
\[
  w(K) := \E_{\theta \sim S^{n-1}}[\max\{ \left<\theta,x-y\right> : x,y \in K\}].
\]
For a compact convex  $K \subseteq \setR^n$ with $\bm{0} \in \textrm{rel.int}(K)$ we denote its \emph{polar} by
\[
  K^\circ := \{ y \in \textrm{span}(K) : \langle x, y \rangle \le 1 \ \forall x \in K\}.
\]
For illustration, suppose $K = \textrm{conv}\{ a_1,\dots,a_m\}$ is a polytope with $\bm{0}$ in its interior. Then the polar is described by inequalities with normal vectors $a_i$, i.e. $K^{\circ} = \{ x \in \setR^n \mid \left<a_i,x\right> \leq 1 \; \forall i\in [m]\}$.
We recall the following basic facts.
\begin{lemma}[Properties of polarity] \label{lem:PropertiesOfPolarity}
  For two convex bodies $K,P \subseteq \setR^n$ with $\bm{0} \in \textrm{int}(K)$ and $\bm{0} \in \textrm{int}(P)$ the following hold: 
  \begin{enumerate}
  \item[(a)]  $(K^{\circ})^{\circ} = K$.
 \item[(b)] For any subspace $F \subseteq \setR^n$ one has $\Pi_F(K)^{\circ} = K^{\circ} \cap F$.
 \item[(c)]  $(K \cap P)^{\circ} = \textrm{conv}(K^{\circ} \cup P^{\circ})$.
 \item[(d)]  $(-K)^{\circ} = -K^{\circ}$.
 \end{enumerate}
\end{lemma}

We write $N(\bm{0},I_n)$ as the standard Gaussian distribution on $\setR^n$.  The \emph{$\ell$-value} of a symmetric convex $Q \subseteq \setR^n$ is defined as
\[
  \ell_Q = \E_{x \sim N(\bm{0},I_n)}[\|x\|_Q^2]^{1/2}.
\]
One may think of $\ell_Q$ as the ``average thinness'' of $Q$. It turns out that the $\ell$-value is also related to the mean width. To see this, note that $\| \cdot \|_{Q^{\circ}}$ is the \emph{dual norm}
to $\| \cdot \|_Q$, i.e. for all $x \in \setR^n$ one has
$
  \|x\|_{Q^{\circ}} = \max\{ \left<x,y\right> :  y \in Q \}
$.
Then
\begin{equation} \label{eq:LQPolarVsWidth}
  \ell_{Q^{\circ}} = \E_{x \sim N(\bm{0},I_n)}[\|x\|_{Q^{\circ}}^2]^{1/2} = \E_{\bm{x} \sim N(\bm{0},I_n)}\big[\max\{ \left<x,y\right>^2 : y \in Q\}\big]^{1/2}.
\end{equation}
We can see that the right hand side of \eqref{eq:LQPolarVsWidth} almost matches the definition of $w(Q)$. In fact, one can prove:
\begin{lemma} \label{lem:LValueVsMeanWidth}
  For any symmetric convex body $Q \subseteq \setR^n$ one has  $
     \ell_{Q^{\circ}} \asymp \sqrt{n} \cdot w(Q).
  $
\end{lemma}
For a positive semidefinite matrix $\Sigma$
we write $N(\bm{0},\Sigma)$ as the Gaussian with mean $\bm{0}$ and covariance matrix $\Sigma$
and for a subspace $U \subseteq \setR^n$ we write $I_U$ as the identity matrix on that subspace. 
Occasionally we will need to refer to the $\ell$-value of a compact symmetric convex set $Q$ that
is not necessarily full-dimensional. In that case we extend the definition to $ \ell_Q = \E_{x \sim N(\bm{0},I_{\textrm{span}(Q)})}[\|x\|_Q^2]^{1/2}$.
In other words, we think of $Q$ as the $\dim(\textrm{span}(Q))$-dimensional convex body whose
ambient space is $\textrm{span}(Q)$. We also note that $\| \cdot \|_Q$ is a norm in $\textrm{span}(Q)$
and not defined outside of $\textrm{span}(Q)$.

We say that a symmetric convex body $Q$ is in \emph{$\ell$-position} if $\ell_Q \cdot \ell_{Q^{\circ}} \leq O(n \log(2n))$.
A powerful tool in convex geometry is that every symmetric convex body can indeed be brought into $\ell$-position: 
\begin{theorem}[Figiel and Tomczak-Jaegermann~\cite{FigielTomczakJaegermann1979}, Lewis~\cite{LewisEllipsoids1979}, Pisier~\cite{PisierKConvexity1982}] \label{thm:PisierRescaling}
For any symmetric convex body $Q \subseteq \setR^n$, there is an invertible linear map $T : \setR^n \to \setR^n$ so that $\ell_{T(Q)} \cdot \ell_{(T(Q))^{\circ}} \leq O(n\log(2n))$. 
\end{theorem}
By Lemma~\ref{lem:LValueVsMeanWidth}, the conclusion of Theorem~\ref{thm:PisierRescaling} is equivalent to
$w(T(Q)) \cdot w(T(Q)^{\circ}) \leq O(\log(2n))$. Moreover one can prove that for any symmetric convex body $Q$ one has $w(Q) \cdot w(Q^{\circ}) \gtrsim w(B_2^n)^2 \gtrsim 1$. Then one can interpret Theorem~\ref{thm:PisierRescaling} as the fact that every symmetric convex body can be linearly transformed so that in terms of mean width and
average thinness it is within a $O(\log(2n))$-factor of the Euclidean ball. For the sake of comparison,
we note that the bound that could be obtained via the more classical John's Theorem~\cite{John1948} would be of the order of $\sqrt{n}$. 

We state two estimates concerning monotonicity of the $\ell$-value that will be crucial for our later arguments:
\begin{lemma}[Monotonicity of the $\ell$-value I] \label{lem:MonotonicityLValue}
  Let $Q \subseteq \setR^n$ be a symmetric convex body. Then for any subspace $U \subseteq \setR^n$, one has $\ell_{Q \cap U} \leq \ell_Q$.
\end{lemma}
\begin{proof}
  Indeed, one has
\[
  \ell^2_Q = \mathbb{E}_{z \sim N(\bm{0}, I_U)}\big[ \mathbb{E}_{y \sim N(\bm{0}, I_{U^\perp})} [\|z + y\|_Q^2]\big] \geq \mathbb{E}_{z \sim N(\bm{0}, I_U)} \big[\big\| z + \underbrace{\mathbb{E}_{y \sim N(\bm{0}, I_{U^\perp})} [y]}_{=\bm{0}}\big\|_Q^2\big] = \ell^2_{Q \cap U},
\]
where the inequality follows from Jensen's inequality and the convexity of $y \mapsto \|z + y\|_Q^2$.
\end{proof}

\begin{lemma}[Monotonicity of the $\ell$-value II] \label{lem:LvalueOfProjectedIntersection} 
  Let $Q \subseteq \setR^n$ be a symmetric convex body. For any subspaces $V \subset W \subseteq \setR^n$,  
  one has $\ell_{\Pi_{V^{\perp}}(Q \cap W)} \le \ell_Q$.
\end{lemma}
\begin{proof}
  We have  $\ell_{\Pi_{V^{\perp}}(Q \cap W)} \le \ell_{Q \cap W \cap V^{\perp}} \le \ell_Q$
  using that $\Pi_{V^{\perp}}(Q \cap W) \supseteq Q \cap W \cap V^{\perp}$ and using Lemma~\ref{lem:MonotonicityLValue}. %
\end{proof}

The following classical result says that among all bodies with identical volume, the Euclidean ball minimizes the mean width. 
\begin{theorem}[Urysohn Inequality I] \label{thm:UrysohnInequality}
  For any convex body $K \subseteq \setR^n$ one has
  \[
    w(K) \geq 2 \cdot \Big(\frac{\Vol_n(K)}{\Vol_n(B_2^n)}\Big)^{1/n}.
  \]
\end{theorem}
A slight variant of this inequality will be handy for us:
\begin{corollary}[Urysohn Inequality II] \label{cor:UrysohnInequalityII}
For any symmetric convex body $Q \subseteq \setR^n$ one has $\Vol_n(Q)^{1/n} \lesssim \frac{\ell_{Q^{\circ}}}{n}$.
\end{corollary}
\begin{proof}
  Applying Urysohn's Inequality I we obtain
  \[
 \Vol_n(Q)^{1/n} \stackrel{\textrm{Thm~\ref{thm:UrysohnInequality}}}{\lesssim} w(Q) \cdot \underbrace{\Vol_n(B_2^n)^{1/n}}_{\lesssim 1/\sqrt{n}} \stackrel{\textrm{Lem~\ref{lem:LValueVsMeanWidth}}}{\lesssim} \frac{\ell_{Q^{\circ}}}{n}.
\]
Here we use in particular that $\Vol_n(B_2^n) \leq (\frac{2e}{\sqrt{n}})^n$.
\end{proof}
For a symmetric convex body $Q \subseteq \setR^n$, one can prove that the \emph{Mahler product}
$\textrm{Vol}_n(Q) \cdot \textrm{Vol}_n(Q^{\circ})$ is always in a surprisingly small range,
see e.g. \cite[Chapter 8]{AsymptoticGeometricAnalysis-Book2015}:
\begin{theorem}[Blaschke-Santal\'o-Bourgain-Milman] \label{thm:BSBM}
  For any symmetric convex body $Q \subseteq \setR^n$ one has
  \[
    C_1^n \leq  \frac{\Vol_n(Q) \cdot \Vol_n(Q^{\circ})}{\Vol_n(B_2^n)^2} \leq C_2^n,
  \]
  where $C_1,C_2>0$ are constants.
\end{theorem}

Let $b(K) := \frac{1}{\Vol_n(K)}\int_K x \; dx$ denote the \emph{barycenter} or \emph{centroid} of a convex body $K$. 
Given a convex body $K$, we abbreviate its \emph{inner symmetrizer} $K_{\textrm{sym}} := K \cap -K$. We will run into the issue that we need to control the volume of a non-symmetric convex body $K$,
but Theorem~\ref{thm:PisierRescaling} only holds for symmetric ones. A popular strategy in
convex geometry is to translate $K$ so that $b(K) = \bm{0}$ and then consider $K_{\textrm{sym}}$ which is a symmetric convex body contained in $K$ capturing much of its geometry. The following is a classical result in convex geometry: 
\begin{theorem}[Milman, Pajor~\cite{MilmanPajor2000}] \label{thm:MilmanPajorInequality} 
  For any convex body $K \subseteq \setR^n$ with the barycenter at the origin one has $\Vol_n(K_{\textrm{sym}}) \geq 2^{-n}\Vol_n(K)$.
\end{theorem}

We will actually require a Milman-Pajor type inequality for the \emph{projection} of a convex body $K$
and its symmetrizer into some subspace $F$. Unfortunately, the symmetrizer of the projection is not the same
as the projection of the symmetrizer! Luckily, a good enough approximate version still holds 
which was proven only recently by Vritsiou~\cite{vritsiou2023regular} in the context of showing the existence of regular $M$-ellipsoids for non-symmetric convex bodies.  
\begin{proposition}[Volume of projections I~{\cite[Corollary 11]{vritsiou2023regular}}] \label{prop:VolumeOfProjectionVsSymmetrizer}
Let $K \subseteq \setR^n$ be a convex body so that $b(K) = \bm{0}$ and let $F \subseteq \setR^n$ be a $d$-dimensional subspace. 
Then 
\[
 \Vol_d(\Pi_F(K))^{1/d} \lesssim \Big(\frac{n}{d}\Big)^5 \cdot \log\Big(\frac{en}{d}\Big)^2  \cdot \Vol_d(\Pi_F(K_{\textrm{sym}}))^{1/d}
\]
  Moreover, this remains true if instead of $b(K) = \bm{0}$ one has $\|b(K)\|_{K-K} \leq \frac{1}{2(n+1)}$.
\end{proposition}
To be precise, the ``moreover'' part is not in \cite{vritsiou2023regular} but can be easily added; we give
a proof of that fact in Section~\ref{sec:FindingSubspace}.
In an earlier version of this manuscript, we had shown an inequality with a better factor of $O((\frac{n}{d})^3)$ when the body is
centered so that $b(K^{\circ}) = \bm{0}$, i.e. the origin is the \emph{Santal\'o point} of $K$.
This only affects the implicit constant in Theorem~\ref{thm:KLConj} and it appears more
natural to work with the barycenter; hence we choose to work with Vritsiou's estimate.
For the interested reader, the bound with the Santal\'o point as center can be found in \cite{reis2024subspaceflatnessconjecturefasterV2} and also independently in~\cite{vritsiou2023regular}.

We also prove a custom-tailored inequality for later. For the remainder of this manuscript, $K_{\textrm{sym}}^{\circ}$ always refers to the polar of the symmetrizer, i.e. $(K_{\textrm{sym}})^{\circ}$.
\begin{lemma}[Volume of projections II] \label{lem:VolOfProjVsLValueOfSymmetrizer}
  Let $K \subseteq \setR^n$ be a convex body with $b(K) = \bm{0}$ and let $F \subseteq \setR^n$ be a $d$-dimensional subspace. Then
  \[(\Vol_d(\Pi_F(K)))^{1/d} \lesssim \Big(\frac{n}{d}\Big)^{6} \cdot \frac{\ell_{K_{\textrm{sym}}^{\circ}}}{d}.\]
  Moreover, this remains true if instead of $b(K) = \bm{0}$ one has $\|b(K)\|_{K-K} \leq \frac{1}{2(n+1)}$.
\end{lemma}
\begin{proof}
  Using the volume estimate from Proposition~\ref{prop:VolumeOfProjectionVsSymmetrizer} with the assumption that the barycenter of $K$ lies at the origin (or as a weaker assumption at least $\|b(K)\|_{K-K} \leq \frac{1}{2(n+1)}$), we obtain
  \begin{eqnarray*}
    (\Vol_d(\Pi_F(K)))^{1/d} &\stackrel{\textrm{Prop~\ref{prop:VolumeOfProjectionVsSymmetrizer}}}{\lesssim}& \Big(\frac{n}{d}\Big)^6 \cdot (\Vol_d(\Pi_F(K_{\textrm{sym}})))^{1/d} \\
    &\stackrel{\textrm{Cor~\ref{cor:UrysohnInequalityII}}}{\lesssim}& \Big(\frac{n}{d}\Big)^6 \cdot \frac{\ell_{(\Pi_F(K_{\textrm{sym}}))^{\circ}}}{d} \\
    &=& \Big(\frac{n}{d}\Big)^{6} \cdot \frac{\ell_{K_{\textrm{sym}}^{\circ} \cap F}}{d} \\ &\stackrel{\textrm{Lem~\ref{lem:MonotonicityLValue}}}{\leq}& \Big(\frac{n}{d}\Big)^{6} \cdot \frac{\ell_{K_{\textrm{sym}}^{\circ}}}{d}.
  \end{eqnarray*}
  Here we also used the fact that $(\Pi_F(K_{\textrm{sym}}))^{\circ} = K_{\textrm{sym}}^{\circ} \cap F$ by Lemma~\ref{lem:PropertiesOfPolarity}. 
\end{proof}

\subsection{Properties of the covering radius}

While the set $K$ may not be symmetric, the sets $\Lambda$ and $\setR^n$ are symmetric, which implies
the following:
\begin{lemma}[Properties of the covering radius~{\cite[Page 31]{DadushThesis2012}}] \label{lem:PropertiesOfCoveringRadius}
  Consider a lattice $\Lambda \subseteq \setR^n$ and a compact convex set $K \subseteq \setR^n$ with $\textrm{span}(\Lambda) = \textrm{affine.hull}(K)$. Then
  \begin{enumerate}
  \item[(a)] $\mu(\Lambda,K) = \mu(\Lambda,K+u)$ for all $u \in \mathrm{span}(\Lambda)$.
  \item[(b)] $\mu(\Lambda,K) = \min\{ r \geq 0 \mid (x + rK) \cap \Lambda \neq \emptyset \; \forall x \in \mathrm{span}(\Lambda)\}$.
  \end{enumerate}
\end{lemma}

We need a triangle inequality for the covering radius: 
\begin{lemma}[Triangle inequality for covering radius] \label{lem:TriangleIneqGenCoveringRadius}
  Let $\Lambda \subseteq \setR^n$ be a lattice and let $\Lambda' \subseteq \Lambda$ be a primitive sublattice.
  Then for any compact convex set $K \subseteq \setR^n$ with $\bm{0} \in \textrm{rel.int}(K)$ and $\textrm{span}(\Lambda)=\textrm{span}(K)$ one has
  \[
    \mu(\Lambda,K) \leq  \mu(\Lambda',K \cap W) + \mu(\Lambda / \Lambda',\Pi_{W^{\perp}}(K)), 
  \]
  where $W := \textrm{span}(\Lambda')$. 
\end{lemma}
\begin{proof}
  W.l.o.g. we may assume that $\Lambda$ has full rank, so $\bm{0} \in \textrm{int}(K)$.
  Following the characterization in Lemma~\ref{lem:PropertiesOfCoveringRadius}.(b), we fix an $x \in \setR^n$.
  For $r_1 := \mu(\Pi_{W^{\perp}}(\Lambda),\Pi_{W^{\perp}}(K))$
  we know that  $\Pi_{W^{\perp}}(x + r_1K) \cap \Pi_{W^{\perp}}(\Lambda) \neq \emptyset$. That means there is a $u_1 \in r_1K$ and a lattice point $y \in \Lambda$   so that $\Pi_{W^{\perp}}(x+u_1) = \Pi_{W^{\perp}}(y)$.
  Next, for $r_2 := \mu(\Lambda \cap W, K \cap W)$ we know that $(x+u_1-y + r_2 \cdot (K \cap W)) \cap (\Lambda \cap W) \neq \emptyset$
  which is equivalent to $(x+u_1 + r_2 \cdot (K \cap W)) \cap (y+(\Lambda \cap W)) \neq \emptyset$.
  Let $u_2 \in r_2 \cdot (K \cap W)$ be the vector so that $x+u_1+u_2 \in \Lambda$. Then $u_1+u_2 \in (r_1+r_2)K$
  by convexity, so $(x + (r_1+r_2) \cdot K) \cap \Lambda \neq \emptyset$.
  \iftrue
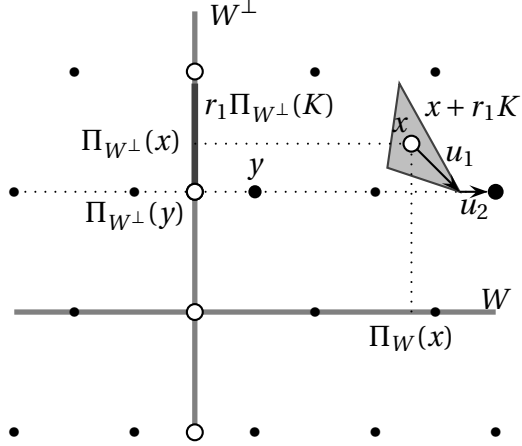
\begin{center} 
\psset{unit=1.6cm}
 \begin{pspicture}(-2,-1)(3,2.5)
    \psline[linewidth=2pt,linecolor=gray](-1.5,0)(2.5,0)\rput[c](2.5,5pt){$W$}
    \psline[linewidth=2pt,linecolor=gray](0,-1)(0,2.5)\rput[l](5pt,2.5){$W^{\perp}$}
    \psline[linestyle=dotted](-1.5,1)(2.5,1)
    \psline[linewidth=2.5pt,linecolor=darkgray](0,1)(0,1.9)\rput[l](3pt,1.7){$r_1\Pi_{W^{\perp}}(K)$}
    \rput[c](1.8,1.4){\pspolygon[fillstyle=solid,fillcolor=lightgray,linecolor=darkgray](0.4,-0.4)(-0.1,0.5)(-0.2,-0.2)}
   \rput[c](0,2){\psdots(-1,0)(0,0)(1,0)(2,0)}
    \psdots(-1,0)(0,0)(1,0)(2,0)
   \psdots(-1.5,1)(-0.5,1)(0.5,1)(1.5,1)(2.5,1)
   \rput[c](0,-2){\psdots(-1.5,1)(-0.5,1)(0.5,1)(1.5,1)(2.5,1)}
   \psline[linestyle=dotted](1.8,1.4)(1.8,0) \rput[c](1.8,-10pt){$\Pi_{W}(x)$}
   \psline[linestyle=dotted](1.8,1.4)(0,1.4) \rput[r](-5pt,1.4){$\Pi_{W^{\perp}}(x)$}
   \cnode[fillstyle=solid,fillcolor=white](1.8,1.4){3pt}{x}\nput[labelsep=2pt]{120}{x}{$x$}
   \multido{\N=-1+1}{4}{\cnode[fillstyle=solid,fillcolor=white](0,\N){3pt}{A}}
   \cnode*(0.5,1){2.5pt}{y2} \nput[labelsep=2pt]{90}{y2}{$y$}
   \cnode[fillstyle=solid,fillcolor=white](0,1){3pt}{p1} \nput[labelsep=2pt]{-135}{p1}{$\Pi_{W^{\perp}}(y)$}
   \pnode(2.2,1){A2}
   \ncline[arrowsize=5pt]{->}{x}{A2}\naput[labelsep=2pt]{$u_1$}
   \cnode*(2.5,1){3pt}{p2}
   \ncline[arrowsize=5pt]{->}{A2}{p2}\nbput[labelsep=2pt]{$u_2$}
   \rput[l](1.9,1.7){$x+r_1K$} 
  \end{pspicture}
\end{center}
\fi
\end{proof}

The natural extension of Lemma~\ref{lem:TriangleIneqGenCoveringRadius} to a filtration is as follows: 
\begin{lemma}[Decomposition of the covering radius] \label{lem:TriangleIneqForFiltration}
  Let $\Lambda \subseteq \setR^n$ be a lattice with any sequence of primitive sublattices
  $\{ \bm{0} \} = \Lambda_0 \subset \Lambda_1 \subset \dots \subset \Lambda_k= \Lambda$. Then for any compact convex set $K \subseteq \setR^n$ with $\bm{0} \in \textrm{rel.int}(K)$ and $\textrm{span}(\Lambda)=\textrm{span}(K)$, one has
  \[
  \mu(\Lambda,K) \leq \sum_{i=1}^k \mu\left( \Lambda_i/\Lambda_{i-1}, \Pi_{\textrm{span}(\Lambda_{i-1})^{\perp}}(K \cap \textrm{span}(\Lambda_i)) \right).
  \]
\end{lemma}
\begin{proof}
We use Lemma~\ref{lem:TriangleIneqGenCoveringRadius} to show by induction over $i_0=k,k-1,\dots,1$ that 
\[\mu(\Lambda,K) \leq \mu(\Lambda_{i_0 - 1}, K \cap \textrm{span}(\Lambda_{i_0-1})) + \sum_{i=i_0}^k \mu\left( \Lambda_i/\Lambda_{i-1}, \Pi_{\textrm{span}(\Lambda_{i-1})^{\perp}}(K \cap \textrm{span}(\Lambda_i)) \right). \]
Indeed, for $i_0 = k$ this is exactly Lemma~\ref{lem:TriangleIneqGenCoveringRadius}. If it holds for some $i_0 > 1$, then
\begin{eqnarray*}
  \mu(\Lambda_{i_0 - 1}, K \cap \textrm{span}(\Lambda_{i_0-1})) &\le& \mu(\Lambda_{i_0 - 2}, K \cap \textrm{span}(\Lambda_{i_0-2})) + \\
                                                                & & \mu\left( \Lambda_{i_0-1}/\Lambda_{i_0-2}, \Pi_{\textrm{span}(\Lambda_{i_0-2})^{\perp}}(K \cap \textrm{span}(\Lambda_{i_0 -1})) \right), 
\end{eqnarray*}
since $\textrm{span}(\Lambda_{i_0-2}) \subset \textrm{span}(\Lambda_{i_0-1})$. So the claim follows by induction, and taking $i_0 := 1$ yields the statement.
\end{proof}

We also need the following fact: 
\begin{lemma}[Monotonicity of $\mu_{KL}$] \label{lem:MonotonicityOfMuKL}
  For any lattice $\Lambda \subseteq \setR^n$, compact convex set $K$ with $\textrm{span}(\Lambda) = \textrm{affine.hull}(K)$
  and subspace $V \subseteq \textrm{span}(\Lambda)$ one has
  $\mu_{KL}(\Pi_{V}(\Lambda),\Pi_{V}(K)) \leq \mu_{KL}(\Lambda,K)$.
\end{lemma}
\begin{proof}
  Let $W \subseteq V$ be the subspace attaining the left side with $\dim W = d$. Then
  \[
 \mu_{KL}(\Pi_{V}(\Lambda),\Pi_{V}(K)) = \Big(\frac{\det(\Pi_W(\Pi_V(\Lambda)))}{\Vol_d(\Pi_W(\Pi_V(K)))}\Big)^{1/d} = \Big(\frac{\det(\Pi_W(\Lambda))}{\Vol_d(\Pi_W(K))}\Big)^{1/d} \leq \mu_{KL}(\Lambda,K),
  \]
  using that $\Pi_W(\Pi_V(x)) = \Pi_W(x)$ for all $x \in \setR^n$ as  $W \subseteq V$.
\end{proof}

\subsection{Approximately stable lattices and the covering radius}

Already Dadush and Regev~\cite{TowardsReverseMinkowskiDadushRegevFOCS16} observed that the Reverse Minkowski Theorem implies a strong upper bound on the
covering radius of a stable lattice. This is based on the following result of Banaszczyk that almost all Gaussian weight is concentrated on lattice points of length at most $O(\sqrt{n})$:
\begin{lemma}[{\cite[Lemma 1.5.(ii)]{TransferenceTheorems-Banaszczyk93}}] \label{lem:HoleFromEuclideanBall}
  For any full rank lattice $\Lambda \subseteq \setR^n$, $c \geq \frac{1}{2}$ and any $u \in \setR^n$ one has $\rho_1( (\Lambda + u) \setminus c \sqrt{n} B_2^n) \leq 2^{-c^2 n} \cdot \rho_1(\Lambda)$.
\end{lemma}
This can then be used as follows: 
\begin{lemma}[Covering radius for stable lattices I~\cite{TowardsReverseMinkowskiDadushRegevFOCS16,Regev-SD-ReverseMinkowskiTheoremSTOC17}] \label{lem:CovRadiusForStableLatticeWithBall}
For any stable lattice $\Lambda \subseteq \setR^n$ one has  $\mu(\Lambda,B_2^n) \leq O(\sqrt{n} \log(2n))$.
\end{lemma}
\begin{proof}
  Let $s := C\log(2n)$ be the parameter from the Reverse Minkowski Theorem (Theorem~\ref{thm:ReverseMinkowskiTheorem}). By Lemma~\ref{lem:PropertiesOf2StableLattices}.(c) (whose proof uses the Reverse Minkowski Theorem) we know that $\rho_s(\Lambda + u) \geq \frac{1}{3}\rho_s(\Lambda)$ for all translates $u$.  
 Now suppose that the lattice $\Lambda$ has a ``hole'' of radius $sc\sqrt{n}$ with center $u$, i.e. $\Lambda \cap (u + sc\sqrt{n} B_2^n) = \emptyset$. Then
  \[
    \frac{1}{3} \rho_s(\Lambda) \leq \rho_s(\Lambda+u) \stackrel{\textrm{hole}}{=} \rho_s((\Lambda+u) \setminus cs\sqrt{n}B_2^n) \stackrel{\textrm{Lem~\ref{lem:HoleFromEuclideanBall}}}{\leq} 2^{-c^2 n} \rho_s(\Lambda)
    \]
   which is a contradiction for say $c \geq 2$ and $n \geq 1$.
\end{proof}
For our purpose we need to generalize the covering radius bound to work for $t$-stable lattices and --- most crucially --- for arbitrary symmetric convex bodies. Luckily, the follow-up work of Banaszczyk~\cite{Banaszczyk1996TransferenceTheoremsForGeneralConvexBodies} provides a framework to do exactly that!
For a symmetric convex body $Q \subseteq \setR^n$, we consider the following quantity 
\[
 \beta(Q) = \sup_{\Lambda \subseteq \setR^n\textrm{ lattice}} \sup_{u \in \setR^n} \frac{\rho_1((u + \Lambda) \setminus Q)}{\rho_1(\Lambda)}.
\]
Note that always $0 < \beta(Q) \leq 1$.
As part of a more general Transference Theorem, Banaszczyk showed how to relate the $\ell$-value of a body to its $\beta$-value: 
\begin{lemma}[Banaszczyk~{\cite[Lemma 2]{Banaszczyk1996TransferenceTheoremsForGeneralConvexBodies}}] \label{lem:BanaszczykLKvsBetaK}
  For any $\varepsilon>0$, there is a $\delta > 0$ so that the following holds: for any
 $n \in \setN$ and any symmetric convex body $Q \subseteq \setR^n$ with $\ell_Q \leq \delta$ one has $\beta(Q) \leq \varepsilon$.
\end{lemma}
For example for the Euclidean ball we have $\ell_{B_2^n} = \sqrt{n}$ and so for any $\varepsilon > 0$ there is a constant $c>0$ so that
$\ell_{c\sqrt{n}B_2^n} = \frac{1}{c}$ is small enough and so $\rho_1((u+\Lambda) \setminus c\sqrt{n}B_2^n) \leq \varepsilon \cdot \rho_1(\Lambda)$ which is consistent with what we know from Lemma~\ref{lem:HoleFromEuclideanBall}.
\begin{proposition}[Covering radius for stable lattices II] \label{prop:CovRadiusOf2StableLattice}
  Let $\Lambda \subseteq \setR^n$ be a full rank lattice such that $\frac{1}{r} \Lambda$ is $t$-stable and let $Q \subseteq \setR^n$ be a symmetric convex body.
  Then $\mu(\Lambda,Q) \leq O(\log(2n)) \cdot t \cdot r \cdot \ell_Q$.
\end{proposition}
\begin{proof}
  Let $\varepsilon>0$ be a small enough constant that we determine later.
  Let $\delta$ be the constant so that Lemma~\ref{lem:BanaszczykLKvsBetaK} applies (w.r.t. $\varepsilon$).
  The claim is invariant under scaling $Q$, hence we may scale $Q$ so that $\ell_Q \leq \delta$ and consequently $\beta(Q) \leq \varepsilon$.
  We may also scale the lattice so that $\Lambda$ is $t$-stable (i.e. $r=1$). 
  It suffices to prove that under these assumptions,  $\mu(\Lambda,Q) \leq s \cdot t$ where $s := C\log(2n)$
  is the parameter from Lemma~\ref{lem:PropertiesOf2StableLattices}. 
  Now suppose for the sake of contradiction that there is a translate $u \in \setR^n$ so that $(u+\Lambda) \cap stQ = \emptyset$.
  Since $\beta(Q) \leq \varepsilon$, we know that
  \[
    \rho_1\Big( \Big(\frac{u}{st} + \frac{\Lambda}{st}\Big) \setminus Q\Big) \leq \varepsilon \rho_1\Big(\frac{\Lambda}{st}\Big).
  \]
 Multiplying the sets and parameters by $st$ gives
  \[
 \rho_{st}((u+\Lambda) \setminus stQ) \leq \varepsilon \rho_{st}(\Lambda). \quad (*)
  \]
 Using that $\Lambda$ is $t$-stable, we get
  \[
\frac{1}{3}\rho_{st}(\Lambda) \stackrel{\textrm{Lem~\ref{lem:PropertiesOf2StableLattices}}}{\leq} \rho_{st}(u+\Lambda) \stackrel{(u+\Lambda) \cap stQ = \emptyset}{=} \rho_{st}((u + \Lambda) \setminus stQ) \stackrel{(*)}{\leq} \varepsilon \rho_{st}(\Lambda).
  \]
  Then choosing $\varepsilon \in (0, \frac{1}{3})$ gives a contradiction. \qedhere
%
\end{proof}
We note that already for the integer lattice and the Euclidean ball one has $\mu(\setZ^n,B_2^n) = \frac{\sqrt{n}}{2}$. Hence Prop~\ref{prop:CovRadiusOf2StableLattice} is almost tight, except (possibly) the logarithmic factor.

Looking forward, the main challenge is to extend the statement of Prop~\ref{prop:CovRadiusOf2StableLattice} to all lattices using the canonical filtration.

\section{Analyzing the canonical filtration} \label{sec:AnalyzingCanonicalFiltration}

First, we want to prove the inequality from Theorem~\ref{thm:KLConj} (with an even better exponent)
in the special case that the lattice is approximately stable and the body $K$ is in $\ell$-position. We will not actually
use Prop~\ref{prop:MainIneqforWellScaledCase} later in this form, but it will eventually guide us towards a general proof strategy.
\begin{proposition} \label{prop:MainIneqforWellScaledCase}
  Let $\Lambda \subseteq \setR^n$ be a full rank $t$-stable lattice with $\det(\Lambda)=1$ and let $K$ be a convex body with $b(K)=\bm{0}$ 
  so that $K_{\textrm{sym}}$ is in $\ell$-position. Then $\mu(\Lambda,K) \leq O(t\log^2(2n)) \cdot R_{\setR^n}(\Lambda,K)$ and 
  in particular, $\mu(\Lambda,K) \leq O(t \log^2(2n)) \cdot \mu_{KL}(\Lambda,K)$.
\end{proposition}
\begin{proof}
  We can upper bound the covering radius by applying the estimate for stable lattices
  from Prop~\ref{prop:CovRadiusOf2StableLattice}:
  \[
   \mu(\Lambda,K) \stackrel{K \supseteq K_{\textrm{sym}}}{\leq} \mu(\Lambda,K_{\textrm{sym}}) \stackrel{\textrm{Prop~\ref{prop:CovRadiusOf2StableLattice}}}{\lesssim} t\log(2n) \cdot \ell_{K_{\textrm{sym}}}
  \]
  Next, we lower bound $\mu_{KL}(\Lambda,K)$ by simply choosing the subspace $W := \setR^n$ as witness. Then
  \[
   \mu_{KL}(\Lambda,K) \geq R_{\setR^n}(\Lambda,K) = \Big(\frac{\det(\Lambda)}{\Vol_n(K)}\Big)^{1/n} \stackrel{(**)}{\gtrsim} \frac{1}{\Vol_n(K_{\textrm{sym}})^{1/n}} \stackrel{\textrm{Cor~\ref{cor:UrysohnInequalityII}}}{\gtrsim} \frac{n}{\ell_{K_{\textrm{sym}}^{\circ}}} \stackrel{\ell\textrm{-position}}{\gtrsim} \frac{\ell_{K_{\textrm{sym}}}}{\log(2n)},
 \]
 where we use in $(**)$ that $\Vol_n(K_{\textrm{sym}}) \geq 2^{-n}\Vol_n(K)$ by Theorem~\ref{thm:MilmanPajorInequality}.
  Combining both inequalities gives the claim.
\end{proof}

The statement of Theorem~\ref{thm:KLConj} is invariant under applying the same linear transformation to both
$\Lambda$ and $K$ and hence we may assume w.l.o.g. that the symmetrizer of $K$ is in $\ell$-position. Of course, the lattice $\Lambda$ may not be (approximately) stable. But we know from Def~\ref{def:tStableFiltration} that the quotient lattices of a $t$-stable filtration are indeed $t$-stable, and hence their covering radius can be analyzed
by generalizing the proof of Prop~\ref{prop:MainIneqforWellScaledCase}.
Additionally, we can analyze the lower bounds on $\mu_{KL}(\Lambda,K)$ that can be derived
from subspaces arising from the filtration.
\begin{lemma} \label{lem:AnalyzingTheCanonicalFiltration}
  Let $K \subseteq \setR^n$ be a convex body with $b(K) = \bm{0}$. Let
  $\{ \bm{0} \} = \Lambda_0 \subseteq \dots \subseteq \Lambda_k = \Lambda$ be a well-separated $t$-stable filtration of a full rank lattice $\Lambda \subseteq \setR^n$. Denote $d_{i} := \textrm{rank}(\Lambda_i / \Lambda_{i-1})$ and $r_i := \det(\Lambda_i/\Lambda_{i-1})^{1/d_i}$.  Then
  \begin{enumerate}
  \item[(I)] For all $i \in [k]$ one has $\mu(\Lambda_i, K_{\textrm{sym}} \cap \textrm{span}(\Lambda_i)) \lesssim t \log(2n) \cdot r_i \cdot \ell_{K_{\textrm{sym}}}$.
  \item[(II)] For any $i \in [k]$ one has $\mu_{KL}(\Lambda,K) \geq R_W(\Lambda,K) \gtrsim r_i \cdot (\frac{d_{\geq i}}{n})^6 \cdot \frac{d_{\geq i}}{\ell_{K_{\textrm{sym}}^{\circ}}}$ where $d_{\geq i} := d_i + \dots + d_k$ and  $W := \textrm{span}(\Lambda_{i-1})^{\perp}$ with $\dim(W)=d_{\geq i}$.
  \end{enumerate}
  Moreover, (I) and (II) remain both true if instead of $b(K) = \bm{0}$ one has $\|b(K)\|_{K-K} \leq \frac{1}{2(n+1)}$.
\end{lemma}

\begin{proof}
  We prove both parts separately. \\
{\bf Part (I).}  
We abbreviate $K_j := \Pi_{\textrm{span}(\Lambda_{j-1})^{\perp}}(K_{\textrm{sym}} \cap \textrm{span}(\Lambda_j))$ for $j \in [k]$. Then $K_j$ is convex and symmetric and $\frac{1}{r_j}(\Lambda_j/\Lambda_{j-1})$ is a $t$-stable lattice. Hence we can bound the covering radii of the
individual quotient lattices by
  \begin{equation} \label{eq:CoveringRadiusQuotientLattice}
  \mu\big(\Lambda_j/\Lambda_{j-1},K_j\big) \stackrel{\textrm{Prop~\ref{prop:CovRadiusOf2StableLattice}}}{\lesssim} \log(2n) \cdot t \cdot r_j \cdot \ell_{K_j} \stackrel{\textrm{Lem~\ref{lem:LvalueOfProjectedIntersection}}}{\leq} \log(2n) \cdot t \cdot r_j \cdot \ell_{K_{\textrm{sym}}}. 
  \end{equation}
   Then for any $i \in [k]$, using the triangle inequality for the covering radius,
  \begin{eqnarray*}
    \mu(\Lambda_{i}, K_{\textrm{sym}} \cap \textrm{span}(\Lambda_{i})) &\stackrel{\textrm{Lem~\ref{lem:TriangleIneqForFiltration}}}{\leq}& \sum_{j=1}^{i} \mu\left( \Lambda_j / \Lambda_{j-1},K_j \right) \\ 
      &\stackrel{\eqref{eq:CoveringRadiusQuotientLattice}}{\lesssim}& \log (2n) \cdot \ell_{K_{\textrm{sym}}} \cdot t \cdot \sum_{j=1}^{i} r_j  \\ &\lesssim& \log (2n) \cdot \ell_{K_{\textrm{sym}}} \cdot t \cdot r_{i}, 
  \end{eqnarray*}
  where in the last step we use that $r_1 < \dots < r_{i}$ and $r_{j} \leq \frac{1}{2}r_{j+2}$ for all $j$\footnote{We would like to point out that for individual indices $j$, the loss in the inequality $\mu\left( \Lambda_j / \Lambda_{j-1},K_j \right) \lesssim \ell_{K_{\textrm{sym}}} \cdot t \cdot r_j$ may be much larger than the polylogarithmic factor we are aiming for. But the $r_j$'s are geometrically increasing and it suffices if that inequality is reasonably tight for \emph{some} index $j$ where $\mu\left( \Lambda_j / \Lambda_{j-1},K_j \right)$ contributes significantly to the covering radius.}. \qed
  
  {\bf Part (II).} Again, we fix an index $i \in [k]$. As indicated above, we 
  choose the subspace $W := \textrm{span}(\Lambda_{i-1})^{\perp}$ as witness and note that $\Pi_W(\Lambda) = \Lambda / \Lambda_{i-1}$.
  Then the normalized determinant of this quotient lattice is
  \begin{equation} \label{eq:detVsRi}
    \det(\Lambda / \Lambda_{i-1})^{1/d_{\geq i}} = \Big( \prod_{j = i}^k r_j^{d_j}\Big)^{1/\sum_{j = i}^k d_j} \geq r_i,
  \end{equation}
  since the middle expression denotes a geometric average of values $r_i < r_{i+1} < \dots < r_k$. Then
  \begin{eqnarray*}
    R_W(\Lambda,K) = \Big(\frac{\det(\Pi_W(\Lambda))}{\Vol_{d_{\geq i}}(\Pi_W(K))}\Big)^{1/{d_{\geq i}}} 
    \stackrel{\eqref{eq:detVsRi}}{\geq} \frac{r_i}{\Vol_{d_{\geq i}}(\Pi_W(K))^{1/{d_{\geq i}}}} \stackrel{\textrm{Lem~\ref{lem:VolOfProjVsLValueOfSymmetrizer}}}{\gtrsim} r_i\cdot \Big(\frac{d_{\geq i}}{n}\Big)^6 \cdot \frac{d_{\geq i}}{\ell_{K_{\textrm{sym}}^{\circ}}}. \qedhere 
  \end{eqnarray*}
\end{proof}

Let us take a step back and illustrate how the analysis from Prop~\ref{lem:AnalyzingTheCanonicalFiltration} can be used to prove a polylogarithmic KL-type inequality for the case that $K$ is the Euclidean ball. 
Here we will be content with a weaker bound of $O(\log^2(2n))$ which allows us to simplify the arguments of \cite{Regev-SD-ReverseMinkowskiTheoremSTOC17}.
Since projections and slices of Euclidean balls are again Euclidean balls, the set $K_j$ in the proof of Prop~\ref{lem:AnalyzingTheCanonicalFiltration}.(I) is simply a Euclidean ball of radius $1$. As $\ell_{B^{d_j}_2} = \sqrt{d_j}$ for all $j$, we can strengthen Prop~\ref{lem:AnalyzingTheCanonicalFiltration}.(I) for $i=k$ to
\[
  \mu(\Lambda,B_2^n) \lesssim t\log(2n) \cdot \sum_{j=1}^k\sqrt{d_j} r_j
\]
Next, abbreviate the subspace used in the lower bound as $W_i := \textrm{span}(\Lambda_{i-1})^{\perp}$.
We can use that
$\textrm{Vol}_{d_{\geq i}}(\Pi_{W_i}(B_2^n))^{1/d_{\geq i}} = \Vol_{d_{\geq i}}(B_2^{d_{\geq i}})^{1/d_{\geq i}} \asymp \frac{1}{\sqrt{d_{\geq i}}}$, which leads to a strengthening of Prop~\ref{lem:AnalyzingTheCanonicalFiltration}.(II) to
\[
 R_{W_i}(\Lambda,K) \gtrsim \frac{r_i}{\Vol_{d_{\geq i}}(B_{2}^{d_{\geq i}})^{1/d_{\geq i}}} \asymp r_i \cdot \sqrt{d_{\geq i}} \geq r_i \sqrt{d_i}
\]
Then fixing the index $i$ that maximizes $\sqrt{d_i} r_i$ and using that the indices $r_j$ grow geometrically while $d_1 + \dots + d_k = n$, one obtains 
$\sqrt{d_i} r_i \geq \Theta(\frac{1}{\log(n)}) \sum_{j=1}^k \sqrt{d_j} r_j$ and so
\[
 \mu(\Lambda,B_2^n) \lesssim t \log(2n) \sum_{j=1}^k\sqrt{d_j} r_j \lesssim t \log^2(2n) \cdot \sqrt{d_{i}} r_{i} \lesssim t \log^2(2n) \cdot R_{W_{i}}(\Lambda,B_2^n)
\]
Now, let $K$ be an arbitrary convex body and let us try to derive Theorem~\ref{thm:KLConj} from Lemma~\ref{lem:AnalyzingTheCanonicalFiltration} --- but this time without the strengthenings that we used for the Euclidean ball.
In Lemma~\ref{lem:AnalyzingTheCanonicalFiltration}.(I) we would inevitably want to pick $i = k$ so that $\textrm{span}(\Lambda_i) = \setR^n$. In Lemma~\ref{lem:AnalyzingTheCanonicalFiltration}.(II) we leave the choice of $i \in [k]$ open. Then
\[
 \frac{\mu(\Lambda,K)}{\mu_{KL}(\Lambda,K)} \stackrel{\textrm{Lem~\ref{lem:AnalyzingTheCanonicalFiltration}}.(I)+(II)}{\lesssim} t\log(2n) \cdot \frac{r_k}{r_i} \cdot \Big(\frac{n}{d_{\geq i}}\Big)^6 \frac{1}{d_{\geq i}} \cdot \ell_{K_{\textrm{sym}}} \cdot \ell_{K_{\textrm{sym}}^{\circ}} \lesssim t\log(2n)^2 \cdot \Big(\frac{n}{d_{\geq i}}\Big)^7 \cdot \frac{r_k}{r_i}
\]
where the second inequality follows if $K_{\textrm{sym}}$ is in $\ell$-position. But going forward, there may not be an index $i$ so that $\frac{r_k}{r_i}$ is small while $d_{\geq i}$ is close to $n$! This issue is not just an artifact of a careless analysis.
For example, if $K = B_{\infty}^n$ (which is already in $\ell$-position) and $W$ is a $d$-dimensional subspace, $\Vol_d(\Pi_W(K))^{1/d}$ is determined only up to factors of the order  $\sqrt{\frac{n}{d}}$. 
In particular only knowing that  $K_{\textrm{sym}}$ is in $\ell$-position does not provide sufficient information to determine a subspace satisfying Theorem~\ref{thm:KLConj}. 
This is the time to recall that Lemma~\ref{lem:AnalyzingTheCanonicalFiltration}.(I) is more general than we needed so far. In particular, it can be used to upper bound the covering radius for certain \emph{slices} of $K_{\textrm{sym}}$. This can then be combined into an iterative approach to obtain a stronger bound. 

\section{Proof of the main theorem\label{sec:WholeMainProof}}

We will spend the next two subsections proving our main Theorem~\ref{thm:KLConj} by induction over $n$. At each step, we split the lattice $\Lambda$ and the convex body $K$ into a subspace section of dimension at least $n/2$ and a projection onto its orthogonal complement. 

\subsection{The inductive step\label{sec:InductiveStep}}

First, we give a self-contained
description of the inductive step, then later in Section~\ref{sec:MainProof} we describe the main part of the induction.
\begin{proposition} \label{prop:MainArgument}
  There is a universal constant $C_0 > 0$ so that the following holds: 
  For any full rank lattice $\Lambda \subseteq \setR^n$ and any convex body $K \subseteq \setR^n$
  with $b(K) = \bm{0}$,
  there exists a primitive sublattice $\Lambda' \subseteq \Lambda$ with $\rank(\Lambda') \geq n/2$ so that
  \[
    \mu\big(\Lambda', K_{\textrm{sym}} \cap \textrm{span}(\Lambda')\big) \leq  C_0\log^2 (2n) \cdot \mu_{KL}(\Lambda,K).
  \]
\end{proposition}

\begin{proof}
The claim is invariant under applying a linear transformation to $K$ and $\Lambda$. Hence we may assume that $K_{\textrm{sym}}$ is in $\ell$-position, i.e.   $\ell_{K_{\textrm{sym}}} \cdot \ell_{K_{\textrm{sym}}^{\circ}} \leq O(n \log (2n))$.
  Consider a well-separated 2-stable filtration  $\{ \bm{0} \} = \Lambda_{0} \subset \dots \subset \Lambda_k = \Lambda$ given by Corollary~\ref{cor:ExistenceTwoStableWellSeparatedFiltration}.
As before, we use the parameters
\[
  d_i := \rank(\Lambda_i/\Lambda_{i-1}) \quad \textrm{and} \quad r_i := \mathrm{nd}(\Lambda_i/\Lambda_{i-1}) = \det(\Lambda_i/\Lambda_{i-1})^{1/d_i},
\]
which are the rank and normalized determinants of the quotient lattices in the filtration, and we abbreviate $d_{\geq i} := d_i + \dots + d_k$.

Let $i^* \in \{ 1,\dots,k\}$ be the minimal index so that $\textrm{rank}(\Lambda_{i^*}) \geq \frac{n}{2}$ and set $W := \textrm{span}(\Lambda_{i^*-1})^{\perp}$. Note that then $\dim(W) = d_{\geq i^*} = n-\textrm{rank}(\Lambda_{i^*-1}) \geq \frac{n}{2}$ by minimality.
Applying Lemma~\ref{lem:AnalyzingTheCanonicalFiltration} with $t=2$ and index $i^*$ for both parts  (I) and (II) gives
\begin{eqnarray*}
  \frac{\mu(\Lambda_{i^*}, K_{\textrm{sym}} \cap \textrm{span}(\Lambda_{i^*}))}{R_W(\Lambda,K)} &\lesssim& \frac{\log(2n) \cdot r_{i^*} \cdot \ell_{K_{\textrm{sym}}}}{r_{i^*} \cdot (\frac{d_{\geq i^*}}{n})^6 \cdot \frac{d_{\geq i^*}}{\ell_{K_{\textrm{sym}}^{\circ}}}} \\ 
  &\lesssim& \log(2n)^2 \cdot \Big(\frac{n}{d_{\geq i^*}}\Big)^7 \\ &\lesssim& \log(2n)^2.
\end{eqnarray*}
   Hence $\Lambda' := \Lambda_{i^*}$ satisfies the claim.
\end{proof}

\subsection{Completing the main proof\label{sec:MainProof}}

Using Proposition~\ref{prop:MainArgument} we can finish the proof of our main theorem. 
\begin{proof}[Proof of Theorem~\ref{thm:KLConj}]
  Consider a full rank lattice $\Lambda \subseteq \setR^n$ and a convex body $K \subseteq \setR^n$. We will prove by induction over $n$ that 
  \[
  \mu(\Lambda,K) \leq C_0\log^3(2n) \cdot \mu_{KL}(\Lambda,K),
\]
where $C_0 \geq 1$ is the constant from Proposition~\ref{prop:MainArgument}. The
claim is true for $n=1$, hence assume $n \geq 2$ from now on.
The claim is invariant under translations of $K$, hence we may assume that $b(K) = \bm{0}$.
Let $\Lambda' \subseteq \Lambda$ be the primitive sublattice from Prop~\ref{prop:MainArgument} and
set $W := \textrm{span}(\Lambda')$. First suppose that $\dim(W)<n$. Then
  \begin{eqnarray*}
    \mu(\Lambda,K) &\stackrel{\textrm{Lem~\ref{lem:TriangleIneqGenCoveringRadius}}}{\leq}& \mu(\Lambda \cap W,K \cap W) + \mu(\Pi_{W^{\perp}}(\Lambda),\Pi_{W^{\perp}}(K)) \\
    &\stackrel{K \supseteq K_{\textrm{sym}}}{\leq}& \mu(\Lambda \cap W,K_{\textrm{sym}} \cap W) + \mu(\Pi_{W^{\perp}}(\Lambda),\Pi_{W^{\perp}}(K)) \\
                   &\stackrel{\substack{\textrm{Prop~\ref{prop:MainArgument}} \\ +\textrm{ind.}}}{\leq}& C_0\log^2(2n) \cdot \mu_{KL}(\Lambda,K) + C_0\log^3(2\underbrace{\dim(W^{\perp})}_{\leq n/2}) \cdot \underbrace{\mu_{KL}(\Pi_{W^{\perp}}(\Lambda),\Pi_{W^{\perp}}(K))}_{\leq \mu_{KL}(\Lambda,K)} \\
    &\stackrel{\textrm{Lem~\ref{lem:MonotonicityOfMuKL}}}{\leq}& C_0\underbrace{\log^2(2n) \cdot \Big(1+\log(n)\Big)}_{=\log^3(2n)} \cdot \mu_{KL}(\Lambda,K). \qedhere
  \end{eqnarray*}
Finally, if $W = \setR^n$, then the second summand is not needed and $\mu(\Lambda,K) \leq C_0\log^2(2n) \cdot \mu_{KL}(\Lambda,K)$ without the inductive step.
\end{proof}

We note that our induction causes $O(\log(2n))$ many re-centering and rescaling operations using the result of
Figiel, Tomczak-Jaegermann and Pisier (Theorem~\ref{thm:PisierRescaling}). This circumvents the issue that 
the covering radius might be dominated by a subspace of dimension $d$ with $d \ll n$, which may not affect the $\ell$-position of the body sufficiently.

\section{Finding the subspace $W$ in single-exponential time\label{sec:FindingSubspace}}

In this section, we prove Theorem~\ref{thm:FindingSubspaceIn2ToN} which means we need to handle the algorithmic aspects
in finding the subspace $W$ that certifies
the upper bound in our main result, Theorem~\ref{thm:KLConj}. From a complexity viewpoint, the key difficulty lies in finding a $t$-stable
filtration for a given lattice for some parameter $t$. In the proof of Prop~\ref{prop:MainArgument} we have used that
there exists always a $t$-stable filtration for $t = O(1)$. 
While no single-exponential time algorithm is known that could find such a filtration,
Dadush~\cite{Dadush-Finding-DenseLatticeSubspacesSTOC19} proved that a filtration satisfying a factor of $t(n) := \Theta(\log n)$ can indeed be computed in time $2^{O(n)}$.
\begin{theorem}[Theorem 6.4. in~\cite{Dadush-Finding-DenseLatticeSubspacesSTOC19}] \label{thm:StableFiltrationAlgorithm}
Given a lattice $\Lambda := \Lambda(B) \subseteq \setR^n$ one can compute a $O(\log (2n))$-stable filtration of $\Lambda$ in time $2^{O(n)}$ times a polynomial in the encoding length of $B$ with probability at least $1-2^{-\Omega(n)}$.
\end{theorem}
We also use a classical result that approximately computes the volume of a convex body:
\begin{theorem}[\cite{Kannan1997RandomWA}] \label{thm:VolumeComputation}
Given a convex body $K \subseteq \setR^n$ with $(a+r\cdot B^n_2) \subseteq K \subseteq R \cdot B^n_2$ for some $a \in \setR^n$ and a parameter $\varepsilon>0$, there exists a randomized algorithm which outputs a positive number $\zeta$ with $\Vol_n(K)/\zeta \in [1-\varepsilon, 1 + \varepsilon]$. The runtime is polynomial in $n$, $1/\varepsilon$, $\log(1/r)$ and $\log(R)$. 
\end{theorem}
In fact,~\cite{Kannan1997RandomWA} also computes an approximation to the barycenter of $K$:

\begin{theorem}[\cite{Kannan1997RandomWA}] \label{thm:BarycenterComputation}
Given a convex body $K \subseteq \setR^n$ with $(a+r\cdot B^n_2) \subseteq K \subseteq R \cdot B^n_2$ for some $a \in \setR^n$ and a parameter $\delta>0$, there exists a randomized algorithm with running time polynomial in $n$, $\frac{1}{\delta}$, $\log(\frac{1}{r})$ and $\log(R)$, which returns an approximate barycenter $\tilde{x}$ such that $\|b(K) - \tilde{x}\|_{K-K} \le \delta$.
\end{theorem}
For both, Theorem~\ref{thm:VolumeComputation} and Theorem~\ref{thm:BarycenterComputation} we assume access to a separation oracle for $K$.
We note that much progress has been made since the seminal work of Kannan, Lov{\'a}sz and Simonovits~\cite{Kannan1997RandomWA} in order 
to reduce the polynomial running times behind Theorems~\ref{thm:VolumeComputation} and \ref{thm:BarycenterComputation}. As these are lower order terms compared to the exponential running time of our algorithm, we do not elaborate on those further.

We also restate and justify the bound on the volume of projections whose proof we had deferred earlier:
\begin{proposition*}[Proposition~\ref{prop:VolumeOfProjectionVsSymmetrizer} --- Volume of projections I] \label{prop:VolumeOfProjectionVsSymmetrizer}
Let $K \subseteq \setR^n$ be a convex body so that $b(K) = \bm{0}$ and let $F \subseteq \setR^n$ be a $d$-dimensional subspace. 
Then 
\[
 \Vol_d(\Pi_F(K))^{1/d} \lesssim \Big(\frac{n}{d}\Big)^5 \cdot \log\Big(\frac{en}{d}\Big)^2  \cdot \Vol_d(\Pi_F(K_{\textrm{sym}}))^{1/d}
\]
  Moreover, this remains true if instead of $b(K) = \bm{0}$ one has $\|b(K)\|_{K-K} \leq \frac{1}{2(n+1)}$.
\end{proposition*}
\begin{proof}
  The statement without the ``moreover'' part can be found in \cite[Corollary 11]{vritsiou2023regular}.
  It remains to justify that the argument holds if the barycenter is close to $\bm{0}$.
First we recall the following fact which can be found e.g. in \cite[Page 2]{MR4832117}: \\
{\bf Fact.} \emph{For any convex body $S \subseteq \setR^n$ with barycenter at the origin one has $-S \subseteq nS$.} \\
For the actual symmetrizer  $P := (K-b(K)) \cap (-K+b(K))$, we can apply the first part and get
\[ \Vol_d(\Pi_F(K))^{1/d} \lesssim \Big(\frac{n}{d}\Big)^5 \cdot \log\Big(\frac{en}{d}\Big)^2  \cdot  \Vol_d(\Pi_F(P))^{1/d}.\]
Using the fact cited above one has $-2(n+1) b(K) \in K-K \subseteq (n+1) \cdot (K-b(K))$ which can be rearranged to $-b(K) \in K$.
Then $K - b(K) \subseteq K + K = 2K$, so that $P \subseteq 2 K_{\mathrm{sym}}$ and $\Vol_d(\Pi_F(P))^{1/d} \le 2 \cdot \Vol_d(\Pi_F(K_{\mathrm{sym}}))^{1/d}$.
\end{proof}

Next, we prove an algorithmic version of Prop~\ref{prop:MainArgument}.
\begin{proposition} \label{prop:ConstructiveExistenceOfSubspaceLambdaPrime}
  Given a full rank lattice $\Lambda := \Lambda(B) \subseteq \setR^n$ and a convex body $K \subseteq \setR^n$ with $(a+r\cdot B^n_2) \subseteq K \subseteq R \cdot B^n_2$ for some $a \in \setR^n$,  one can compute a primitive sublattice $\Lambda' \subseteq \Lambda$ with  $\textrm{rank}(\Lambda') \geq n/2$, a point $\tilde{x} \in \setR^n$ and a subspace $W \subseteq \setR^n$ with $\dim(W) \geq n/2$ so that
  \begin{equation} \label{eq:ConstructiveGuaranteeOfPartialSubspace0}
   \mu\big(\Lambda', (K - \tilde{x}) \cap (-(K - \tilde{x})) \cap \textrm{span}(\Lambda')\big) \leq C_0t(n)\log^2(2n) \cdot R_{W}(\Lambda,K)
 \end{equation}
  The running time is bounded by a polynomial in $\log(\frac{1}{r})$, $\log(R)$ and the encoding length of $B$,
  assuming access to a separation oracle for $K$ and an oracle to find $t(n)$-stable filtrations.
\end{proposition}

\begin{proof}
To obtain $\Lambda',\tilde{x}$ and $W$, we use the following algorithm which mimics the proof of Prop~\ref{prop:MainArgument}: 
\begin{center}
  \psframebox{\begin{minipage}{14cm}
{\sc Find-Sublattice-for-partial-bound} \\
{\bf Input:} Full rank lattice $\Lambda := \Lambda(B) \subseteq \setR^n$ and a convex body $K \subseteq \setR^n$  \\
{\bf Output:} Sublattice $\Lambda' \subseteq \Lambda$, point $\tilde{x}$ and subspace $W$ satisfying Prop~\ref{prop:ConstructiveExistenceOfSubspaceLambdaPrime}
\begin{enumerate*}
\item[(1)] Compute an approximate barycenter $\tilde{x}$ such that $\|b(K) - \tilde{x}\|_{K-K} \le \frac{1}{2(n+1)}$
\item[(2)] Set $K_{\mathrm{sym}} := (K-\tilde{x}) \cap (-(K-\tilde{x}))$ and compute an invertible linear map $T$ so that
  \[\ell_{T(K_{\mathrm{sym}})} \cdot \ell_{(T (K_{\mathrm{sym}}))^{\circ}} \leq C \cdot n \log (2n) \]
\item[(3)] Set $\bar{K} \gets T(K-\tilde{x})$ and $\bar{\Lambda} \gets T(\Lambda)$
\item[(4)] Compute a $t(n)$-stable filtration  $\{\bm{0}\} = \bar{\Lambda}_0 \subset \dots \subset
      \bar{\Lambda}_k = \bar{\Lambda}$
\item[(5)] Compute a well-separated $2t(n)$-stable filtration  $\{\bm{0}\} = \bar{\Lambda}_0' \subset \dots \subset
      \bar{\Lambda}'_{k'} = \bar{\Lambda}$
    \item[(6)] Set $i^*$ as the minimal index with $\rank(\bar{\Lambda}'_{i^*}) \ge \frac{n}{2}$
\item[(7)] Return $\Lambda' := T^{-1}(\bar{\Lambda}'_{i^*})$, $\tilde{x}$ and $W := \textrm{span}(\Lambda'')^{\perp}$ where $\Lambda'' := T^{-1}(\bar{\Lambda}'_{i^*-1})$
\end{enumerate*}
\end{minipage}}
\end{center}
The choice of $i^*$ gives that  $\textrm{rank}(\bar{\Lambda}'_{i^*-1}) \leq n/2 \leq \textrm{rank}(\bar{\Lambda}'_{i^*})$ which is the same as $\textrm{rank}(\Lambda'') \leq n/2 \leq \textrm{rank}(\Lambda')$ because $T$ is invertible. Then $\dim(W) = n - \textrm{rank}(\Lambda'') \geq n/2$.

Next, we justify the running time. 
The approximate barycenter $\tilde{x}$ in (1) can be computed via Theorem~\ref{thm:BarycenterComputation}. 
The linear map $T$ bringing the body in $\ell$-position in (2) corresponds to solving a convex program
which can be done in randomized polynomial time; for details see e.g. \cite{10.5555/2095116.2095230}. In (4)
we use the assumed oracle to find a $t(n)$-stable filtration which can be refined into a well-separated filtration by Theorem~\ref{thm:ApproxFilt}; the latter argument works in polynomial time as the involved determinants can be computed via Gaussian elimination.

Finally, we prove the main statement. We observe that Eq~\eqref{eq:ConstructiveGuaranteeOfPartialSubspace0} is invariant under an invertible linear transformation and hence we may instead prove that 
\begin{equation} \label{eq:ConstructiveGuaranteeOfPartialSubspace}
   \mu\big(\bar{\Lambda}_{i^*}', T(K_{\textrm{sym}}) \cap \textrm{span}(\bar{\Lambda}_{i^*}')\big) \leq C_0t(n)\log^2(2n) \cdot R_{\textrm{span}(\bar{\Lambda}_{i^*-1}')^{\perp}}(\bar{\Lambda},\bar{K}).
 \end{equation}
 The difference to the earlier proof of Prop~\ref{prop:MainArgument} lies in the use of a well-separated $2t(n)$-stable filtration and the use of an approximate barycenter. First, replacing the factor of ``2'' by ``$2t(n)$'' and
 applying Lemma~\ref{lem:AnalyzingTheCanonicalFiltration}.(I) gives
\begin{equation} \label{eq:ConstructiveGuaranteeOfPartialSubspace2}
 \mu(\bar{\Lambda}_{i^*}', T(K_{\textrm{sym}}) \cap \textrm{span}(\bar{\Lambda}_{i^*}')) \lesssim \log(2n) \cdot t(n) \cdot \ell_{T(K_{\textrm{sym}})} \cdot r_{i^*}.
\end{equation}
where $r_{i^*} := \textrm{nd}(\bar{\Lambda}_{i^*}' / \bar{\Lambda}_{i^*-1}')$.
Similarly, setting $d := \dim(\textrm{span}(\bar{\Lambda}_{i^*-1}')^{\perp}) \geq \frac{n}{2}$ and applying Lemma~\ref{lem:AnalyzingTheCanonicalFiltration}.(II) 
with the fact that  $b(\bar{K})=T(b(K) - \tilde{x})$ and so $\|b(\bar{K})\|_{\bar{K}-\bar{K}} \leq \frac{1}{2(n+1)}$ gives
\begin{equation} \label{eq:ConstructiveGuaranteeOfPartialSubspace3}
 R_{\textrm{span}(\bar{\Lambda}_{i^*-1}')^{\perp}}(\bar{\Lambda},\bar{K}) \gtrsim \frac{r_{i^*}}{\log(2n)} \cdot \Big(\frac{d}{n}\Big)^7 \cdot \ell_{T(K_{\textrm{sym}})}
\end{equation}
Combining Eq~\eqref{eq:ConstructiveGuaranteeOfPartialSubspace2} and Eq~\eqref{eq:ConstructiveGuaranteeOfPartialSubspace3} will then provide Eq~\eqref{eq:ConstructiveGuaranteeOfPartialSubspace}.
\end{proof}

In the following, we will also need to control the encoding length of sublattices. Fortunately,
in order to do so it suffices to control their determinant. 
\begin{lemma} \label{lem:EncodingLengthOfSublatticesInTermsOfDet}
  Let $\Lambda := \Lambda(B) \subseteq \setQ^n$ be a full rank lattice and let $\Lambda' := \Lambda(B') \subseteq \Lambda$ be any sublattice with $d := \textrm{rank}(\Lambda')$.
  Then there is a basis $B^* \in \setQ^{n \times d}$ of $\Lambda'$ that has encoding length polynomial in $n$, $\log(\det(\Lambda'))$ and the encoding length of $B$.
  Such a basis can be found in time polynomial in $n$ and the encoding length of $B'$.
\end{lemma}
\begin{proof}
  After scaling  $\Lambda$ and $\Lambda'$ with the smallest common multiple of the
  denominators in $B$ (whose encoding length is polynomial in $n$ and the encoding length of $B$) we may assume that $B \in \setZ^{n \times n}$.
  We use the LLL-algorithm from Theorem~\ref{thm:LLLalgorithm} to compute an \emph{LLL-reduced basis} $B^* = (b_1,\dots,b_d) \in \setQ^{n \times d}$ of $\Lambda'$ where the running time is polynomial in $n$ and the encoding length of $B'$. From Theorem~\ref{thm:LLLalgorithm}.(iii) we know that such a basis has bounded \emph{orthogonality defect}, which means that
  \[
     \det(\Lambda') \leq \prod_{i=1}^d \|b_i\|_2 \leq 2^{d^2/4} \cdot \det(\Lambda')
  \]
  Since $b_i \in \Lambda' \subseteq \Lambda \subseteq \setZ^n$ we know that $\|b_i\|_2 \geq 1$ for all $i=1,\dots,d$. This in turn implies an upper bound of $\|b_i\|_2 \leq 2^{d^2/4} \det(\Lambda')$ for all $i$. The claim follows.
\end{proof}

Now we can state a constructive generalization of Theorem~\ref{thm:FindingSubspaceIn2ToN} with a general parameter $t(n)$.
Then using the $2^{O(n)}$-time algorithm of Dadush to achieve $t(n) = \Theta(\log n)$ implies Theorem~\ref{thm:FindingSubspaceIn2ToN}.
We would like to point out that the remainder of the presented algorithm will run in polynomial time.

\begin{theorem} \label{thm:ConstructiveFindingFlatSubspace}
Given a full rank lattice $\Lambda := \Lambda(B) \subseteq \setR^n$ and a convex body $K \subseteq \setR^n$ with $(a+r\cdot B^n_2) \subseteq K \subseteq R \cdot B^n_2$ for some $a \in \setR^n$, there is a randomized algorithm to compute a subspace $W \subseteq \setR^n$ so that
\begin{equation} \label{eq:ConstructiveFindingFlatSubspace}
 \mu(\Lambda,K) \leq Ct(n)\log^3(2n) \cdot R_W(\Lambda,K) 
\end{equation}
  The running time is bounded by a polynomial in $n$, $\log(\frac{1}{r})$, $\log(R)$ and the encoding length of $B$
  assuming access to a separation oracle for $K$ and an oracle to find $t(n)$-stable filtrations where $t(n)$ is non-decreasing
  with $t(n) \leq 2^{\textrm{poly}(n)}$.
\end{theorem}
\begin{proof}
The proof presented for Theorem~\ref{thm:KLConj} can easily be interpreted as a recursive algorithm.
The main technical difficulty here is to prove that the encoding length of the involved lattices does
not blow up. In order to prove this, it will be useful to unravel the recursion and phrase the argument
as an iterative method: 
\begin{center}
\psframebox{\begin{minipage}{14cm}
\textsc{Find-Sublattice} \\
{\bf Input:} A full rank lattice $\Lambda := \Lambda(B) \subseteq \setR^n$ and a convex body $K \subseteq \setR^n$ given by a separation oracle. \\
{\bf Output:} Subspace $W \subseteq \setR^n$ satisfying Theorem~\ref{thm:ConstructiveFindingFlatSubspace}
\begin{enumerate*}
\item[(1)] Set $\Lambda_0 := \{ \bm{0}\}$ and $i := 0$
\item[(2)] FOR $i=0$ TO $\infty$ DO 
  \begin{enumerate*}
  \item[(3)] IF $\Lambda_{i}=\Lambda$ THEN set $k := i$ and exit the FOR loop. 
  \item[(4)] Call $\textsc{Find-Sublattice-For-Partial-Bound}(\Lambda / \Lambda_i,  \Pi_{\textrm{span}(\Lambda_i)^{\perp}}(K))$ to find a sublattice $\Lambda_{i}^{\textrm{rec}} \subseteq \Lambda / \Lambda_i$, a point $\tilde{x}_i \in \textrm{span}(\Lambda / \Lambda_i)$ and a subspace $W_i \subseteq \textrm{span}(\Lambda / \Lambda_i)$ with $\dim(W_i) \geq \frac{1}{2} \rank(\Lambda / \Lambda_i)$.
  \item[(5)] Let $\Lambda_{i+1} \subseteq \Lambda$ be the primitive lattice with $\Lambda / \Lambda_{i+1} = (\Lambda / \Lambda_i) / \Lambda_i^{\textrm{rec}}$. 
  \end{enumerate*}
\item[(6)] Return the subspace $W \in \{ W_0,\dots,W_{k-1}\}$ maximizing $R_{W}(\Lambda,K)$.
\end{enumerate*}
\end{minipage}}
\end{center}
Clearly $\Lambda_{i+1} \supseteq \Lambda_i$ for all $i$ and hence the algorithm produces a chain $\{ \bm{0}\} = \Lambda_0 \subseteq  \dots \subseteq \Lambda_k = \Lambda$. Moreover, $\Lambda_i^{\textrm{rec}} = \Lambda_{i+1} / \Lambda_i$.
Since $\textrm{rank}(\Lambda_i^{\textrm{rec}}) \geq \frac{1}{2}\textrm{rank}(\Lambda / \Lambda_i)$, we know that the algorithm terminates after $k \leq O(\log n)$
iterations.

Our goal is to prove that encoding length of all the sublattices $\Lambda_i$ can be bounded, which
by Lemma~\ref{lem:EncodingLengthOfSublatticesInTermsOfDet} reduces to bounding $\det(\Lambda_i)$.
We set a parameter $U := C_0t(n) \log^2(2n) \cdot \mu_{KL}(\Lambda,K)$ where $C_0$ is the constant from Prop~\ref{prop:ConstructiveExistenceOfSubspaceLambdaPrime}.
We note that by assumption, $\log(U)$ is bounded by a polynomial in $n$, $\log(r)$ and the encoding length of $B$.
Abbreviating $K_i := \Pi_{\textrm{span}(\Lambda_i)^{\perp}}(K) - \tilde{x}_i$ and $d_i := \textrm{rank}(\Lambda_{i+1}/\Lambda_i)$ for $i \in \{ 0,\dots,k-1\}$, we can estimate that
\begin{eqnarray}
   \Big(\frac{\det(\Lambda_{i+1})}{\det(\Lambda_i)}\Big)^{1/d_i} \cdot \frac{1}{2R} 
  &\leq& \Big(\frac{\det(\Lambda_{i+1}/\Lambda_i)}{\Vol_{d_i}(R B_2^{d_i})}\Big)^{1/d_i} \label{eq:ConsW2}\\ 
  &\leq& \mu(\Lambda_{i+1} / \Lambda_i,R B_2^{d_i}) \label{eq:ConsW3} \\
  &\leq& \mu(\Lambda_{i+1} / \Lambda_i,K_i \cap (-K_i) \cap \textrm{span}(\Lambda_{i+1} / \Lambda_i))  \label{eq:ConsW4}\\
  &\leq& C_0t(n) \log^2(2n) \cdot R_{W_i}(\Lambda/\Lambda_i,\Pi_{\textrm{span}(\Lambda_i)^{\perp}}(K))  \label{eq:ConsW5} \\
  &=& C_0t(n) \log^2(2n) \cdot R_{W_i}(\Lambda,K)  \label{eq:ConsW6} \\
  &\leq& U \label{eq:ConsW7} 
\end{eqnarray}
Here we use $\Vol_{d_i}(RB_2^{d_i}) \leq (2R)^{d_i}$ and $\det(\Lambda_{i+1}) = \det(\Lambda_i) \cdot \det(\Lambda_{i+1}/\Lambda_i)$ in \eqref{eq:ConsW2}.
We apply the basic determinant / volume lower bound in \eqref{eq:ConsW3}. In \eqref{eq:ConsW4} we use that $K \subseteq R B_2^n$ and projections and intersections cannot increase the radius of a set. In \eqref{eq:ConsW5} we use the
guarantee provided by Prop~\ref{prop:ConstructiveExistenceOfSubspaceLambdaPrime} and in \eqref{eq:ConsW7} we apply the definition of $U$.
Rearranging \eqref{eq:ConsW2}-\eqref{eq:ConsW7} gives the inequality
\begin{equation} \label{eq:ConsWDetB}
 \frac{\det(\Lambda_{i+1})}{\det(\Lambda_i)} \leq (2RU)^{d_i}
\end{equation}
Multiplying \eqref{eq:ConsWDetB} for indices $0,\dots,i-1$ gives
\[
\det(\Lambda_i) =\prod_{j=0}^{i-1} \frac{\det(\Lambda_{j+1})}{\det(\Lambda_j)} \leq \prod_{j=0}^{i-1} (2R U)^{d_j} \leq (2R U)^n
\]
as $\sum_{j=0}^{i-1} d_j \leq n$. By Lemma~\ref{lem:EncodingLengthOfSublatticesInTermsOfDet} we then conclude that $\Lambda_i$ has a basis whose encoding length is bounded by a polynomial in $n$,  $\log(R)$, $\log(U)$ and the encoding length of $B$. Also the quotient lattice $\Lambda / \Lambda_i$ admits a basis whose encoding length is a (fixed) polynomial in the encoding length of the bases of $\Lambda$ and $\Lambda_i$.
Together with the running time bound provided by Prop~\ref{prop:ConstructiveExistenceOfSubspaceLambdaPrime}, this concludes the runtime analysis.

As for the approximation guarantee claimed in \eqref{eq:ConstructiveFindingFlatSubspace} one could either be convinced that the algorithm satisfies the same recursive bound obtained in Theorem~\ref{thm:KLConj}
with an additional $t(n)$ factor, or one could directly verify that via the triangle inequality one has
\begin{eqnarray*}
  \mu(\Lambda,K) &\stackrel{\textrm{Lem~\ref{lem:TriangleIneqForFiltration}}}{\leq}& \sum_{i=0}^{k-1} \mu(\Lambda_{i+1} / \Lambda_i,K_i \cap (-K_i) \cap \textrm{span}(\Lambda_{i+1} / \Lambda_i)) \\
  &\stackrel{\eqref{eq:ConsW5}+\eqref{eq:ConsW6}}{\leq}& C_0t(n) \log^2(2n) \sum_{i=0}^{k-1} R_{W_i}(\Lambda,K) \\
  &\leq& C_0t(n) \log^2(2n) \cdot k \cdot R_W(\Lambda,K)
\end{eqnarray*}
as in step (6) we choose $W$ as the subspace providing the best lower bound. 
\end{proof}

It might be worth pausing at this point and noting that the sequence of subspaces $W_0,\dots,W_{k-1}$ that replace the inductive proof of Theorem~\ref{thm:KLConj} might be more amenable to improve the factor of $O(\log^3 (2n))$ in our main result in the future. This might also be the time to make the origin of the other two logarithmic factors explicit. 
First, let $C_{\textrm{cov}}(n)$ be the infimum over $\alpha > 0$ so that for all $m \leq n$ and all 2-stable full rank lattices $\Lambda \subseteq \setR^m$ and any symmetric convex body $Q \subseteq \setR^m$ one has
$\mu(\Lambda,Q) \leq \alpha \cdot \ell_Q$.
Next, let $C_{\textrm{pos}}(n)$ be the infimum over all $\beta > 0$ so that for all $m \leq n$ and all symmetric convex bodies $Q \subseteq \setR^m$, there is a linear map $T$ with $\ell_{T(Q)} \cdot \ell_{T(Q)^{\circ}} \leq \beta \cdot m$. 
\begin{corollary} \label{cor:KLtypeInequWithSequenceOfSubspaces}
  For any full rank lattice $\Lambda \subseteq \setR^n$ and any convex body $K \subseteq \setR^n$ there is a sequence $W_0,\dots,W_{k-1} \subseteq \setR^n$ of subspaces so that
  \[
   \mu(\Lambda,K) \lesssim C_{\textrm{cov}}(n) \cdot C_{\textrm{pos}}(n) \cdot \sum_{i=0}^{k-1} R_{W_i}(\Lambda,K)
 \]
 where $C_{\textrm{cov}}(n) \leq O(\log(2n))$ and $C_{\textrm{pos}}(n) \leq O(\log(2n))$.
 Moreover, $\dim(W_i) \leq \frac{n}{2^i}$ for $i=0,\dots,k-1$ and $\dim(W_{i+2}) \leq \frac{1}{2} \cdot \dim(W_i)$
 for $i=0,\dots,k-3$.
\end{corollary}
\begin{proof}
  Revisiting Section~\ref{sec:AnalyzingCanonicalFiltration} we can see that the factor of $\log(2n)$ in Lemma~\ref{lem:AnalyzingTheCanonicalFiltration}.(I) can be replaced by $C_{\textrm{cov}}(n)$ and in turn the factor of $\log^2(2n)$
  in Proposition~\ref{prop:MainArgument} can be replaced by $C_{\textrm{cov}}(n) \cdot C_{\textrm{pos}}(n)$.
  We recall that $C_{\textrm{cov}}(n) \leq O(\log(2n))$ by Prop~\ref{prop:CovRadiusOf2StableLattice}
  and $C_{\textrm{pos}}(n) \leq O(\log(2n))$ by Theorem~\ref{thm:PisierRescaling}.
  Next, revisit the proof of Theorem~\ref{thm:ConstructiveFindingFlatSubspace} for $t=1$. The argument produces
  a chain $\{ \bm{0} \} \subseteq \Lambda_0 \subseteq \dots \subseteq \Lambda_k$ of sublattices so that for $n_i := \rank(\Lambda / \Lambda_i)$ one has that
  $n_{i+1} \leq \frac{n_i}{2}$ for $i=0,\dots,k-1$. Moreover, the argument finds subspaces $W_i \subseteq \textrm{span}(\Lambda / \Lambda_i)$ so that
 \[
   \mu(\Lambda,K) \lesssim C_{\textrm{cov}}(n) \cdot C_{\textrm{pos}}(n) \cdot \sum_{i=0}^{k-1} R_{W_i}(\Lambda,K)
 \]
 while $\frac{n_i}{2} \leq \dim(W_i) \leq n_i$.
  One can see by induction that $\dim(W_i) \leq n_i \leq \frac{n}{2^i}$ for $i=0,\dots,k-1$.
  Moreover $\dim(W_{i+2}) \leq n_{i+2} \leq \frac{n_{i}}{4} \leq \frac{1}{2} \dim(W_{i})$ for $i=0,\dots,k-3$
  as claimed.
\end{proof}

\section{Integer programming in time $(\log(2n))^{O(n)}$\label{sec:IP}}

Next, we show that integer programming can be solved in time $(\log(2n))^{O(n)}$. In fact,
this is a known consequence of Theorem~\ref{thm:FindingSubspaceIn2ToN}.
We do not claim any original contribution for this section,
but we reproduce the arguments of Dadush~\cite[Chapter 7]{DadushThesis2012} to be self-contained.

First, we describe the intuition behind Dadush's algorithm. Consider a convex body $K \subseteq \setR^n$
and a lattice $\Lambda \subseteq \setR^n$; the goal is to find a point in $K \cap \Lambda$.
We compute a subspace $W \subseteq \setR^n$ in time $2^{O(n)}$ that certifies the covering radius $\mu(\Lambda,K)$ up to
a factor $\rho(n) := \Theta(\log^{4} (2n))$. 
Consider the points  $X := \Pi_W(K) \cap \Pi_W(\Lambda)$ in the projection on $W$.
For each $x \in K \cap \Lambda$, we also have $\Pi_W(x) \in X$. Note that the reverse may not be true in the
sense that it is entirely possible that $K \cap \Lambda = \emptyset$ while $X \neq \emptyset$.
However, we are guaranteed that all lattice points
in $K$ must be in one of the $(n-d)$-dimensional fibers of the projection, i.e.
\[
 K \cap \Lambda \subseteq \bigcup_{y \in X} \big((K \cap \Pi_W^{-1}(y)) \cap \Lambda\big).
\]
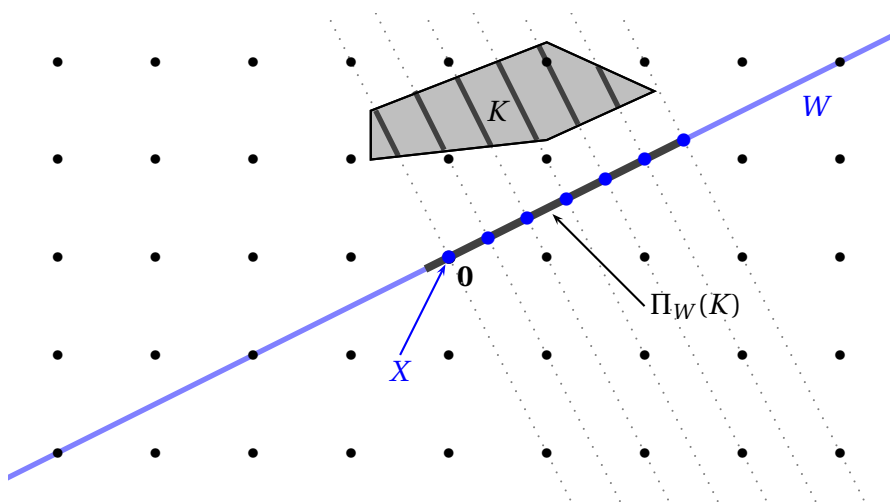
\begin{center}
  \psset{unit=1.3cm}
  \begin{pspicture*}(-4.5,-2.5)(4.5,2.5)
    \pspolygon[fillstyle=solid,fillcolor=lightgray](-0.8,1)(-0.8,1.5)(1,2.2)(2.1,1.7)(1,1.2)\rput[c](0.5,1.5){$K$}%
    \psline[linecolor=blue!50!white,linewidth=2pt](5,2.5)(-5,-2.5)\pnode(3.5,1.75){A}\nput{-45}{A}{$\blue{W}$}%
    \cnode*(0,0){2.5pt}{origin}\nput[labelsep=2pt]{-45}{origin}{$\bm{0}$}%
    \psline[linecolor=darkgray,linewidth=3pt](2.4,1.2)(-0.24,-0.12)%
    \multido{\N=0+1}{4}{\rput[c](\N,0){\psline[linestyle=dotted,linecolor=gray](3,-6)(-3,6)}}%
    \multido{\N=0+1}{3}{\rput[c](\N,1){\psline[linestyle=dotted,linecolor=gray](3,-6)(-3,6)}}%
    \psplot[algebraic=true,linewidth=2pt,linecolor=darkgray]{-0.75}{-0.52}{-2*x}
    \psplot[algebraic=true,linewidth=2pt,linecolor=darkgray]{-0.35}{-0.05}{-2*(x-0.5)}
    \psplot[algebraic=true,linewidth=2pt,linecolor=darkgray]{0.08}{0.43}{-2*(x-1.0)}
    \psplot[algebraic=true,linewidth=2pt,linecolor=darkgray]{0.5}{0.9}{-2*(x-1.5)}
    \psplot[algebraic=true,linewidth=2pt,linecolor=darkgray]{0.92}{1.32}{-2*(x-2.0)}
    \psplot[algebraic=true,linewidth=2pt,linecolor=darkgray]{1.52}{1.73}{-2*(x-2.5)}

    \pspolygon[fillstyle=none,fillcolor=lightgray](-0.8,1)(-0.8,1.5)(1,2.2)(2.1,1.7)(1,1.2)
    \multido{\n=-2+1}{5}{\multido{\N=-4+1}{9}{\psdots(\N,\n)}}%
    \multido{\N=0+0.4,\n=0+0.2}{7}{\psdots[linecolor=blue,linewidth=1.5pt](\N,\n)}
    \pnode(0,0){A}\pnode(-0.5,-1.0){B} \ncline[linecolor=blue,nodesepB=3pt]{->}{B}{A} \nput[labelsep=2pt]{-90}{B}{$\blue{X}$}
    \pnode(1.0,0.5){A}\pnode(2,-0.5){B} \ncline[linecolor=black,nodesepB=3pt]{->}{B}{A} \nput[labelsep=2pt]{0}{B}{$\Pi_{W}(K)$}
\end{pspicture*}
\end{center}
The algorithm enumerates $X$ and then recurses on all the fibers. In order for this algorithm to be
efficient we need to (i) bound the cardinality $|X|$ and (ii) be able to enumerate $X$. For 
(ii), note that it is possible that $W = \setR^n$ and hence we would not gain anything by treating $\Pi_W(K) \cap \Pi_W(\Lambda)$ as a general integer programming problem.

First, we need the insight of Dadush, Peikert and Vempala that enumerating lattice points in an
ellipsoid can be done in time $2^{O(n)}$ per point:
\begin{theorem}[\cite{EnumerateLatticeAlgosDadushPeikertVempalaFOCS11}] \label{thm:EnumerateLatticePointsInBall}
Given a full rank lattice $\Lambda = \Lambda(B)$, an ellipsoid $E$ and a translate $t \in \setQ^n$,
one can compute the set $S := (t + E) \cap \Lambda$ in time $(|S| + 1) \cdot 2^{O(n)}$ times a polynomial
in the encoding length of $B$, $E$ and $t$.
\end{theorem}
\begin{proof}[Proof sketch]
After applying
a linear transformation, it suffices to compute all lattice points in a Euclidean ball, i.e.
we may assume that $S = (t+B_2^n) \cap \Lambda$. Next, consider the \emph{Voronoi cell}
\[
\pazocal{V} = \left\{ x \in \setR^n : \|x\|_2 \leq \|x - v\|_2 \; \forall v \in \Lambda \setminus \{\bm{0}\} \right\}
\]
of the lattice, which are all points that are no farther from the origin than to any
other lattice point. Note that $\pazocal{V}$ is a symmetric convex body; in fact it is a polytope.
Let $R \subseteq \Lambda \setminus \{ \bm{0}\}$ be the \emph{Voronoi-relevant} vectors, which are all the
vectors that define a facet of $\pazocal{V}$. One can prove that  $|R| \leq 2^{n+1}$
and moreover the set $R$ can be computed in time $2^{O(n)}$ by the algorithm of Micciancio and Voulgaris~\cite{CVP-Voronoi-Algo-MicciancioVoulgaris-SICOMP2013}.
Next, consider the graph $H = (\Lambda,F)$ with edges $F = \{ \{ x,y\} : x,y \in \Lambda\textrm{ and }x-y \in R\}$.
\begin{figure}
\begin{center}
\psset{unit=0.8cm}
\begin{pspicture*}(-4,-2.7)(4,3.4)
 \rput[c](5,-2){\pspolygon[fillstyle=solid, fillcolor=lightgray](0.166,1.166)(0.825,0.825)(1.166,0.166)(-0.166,-1.166)(-0.825,-0.825)(-1.16,-0.166)}%
 \rput[c](3,-3){\pspolygon[fillstyle=solid, fillcolor=lightgray](0.166,1.166)(0.825,0.825)(1.166,0.166)(-0.166,-1.166)(-0.825,-0.825)(-1.16,-0.166)}%
 \rput[c](4,-1){\pspolygon[fillstyle=solid, fillcolor=lightgray](0.166,1.166)(0.825,0.825)(1.166,0.166)(-0.166,-1.166)(-0.825,-0.825)(-1.16,-0.166)}%
 \rput[c](2,-2){\pspolygon[fillstyle=solid, fillcolor=lightgray](0.166,1.166)(0.825,0.825)(1.166,0.166)(-0.166,-1.166)(-0.825,-0.825)(-1.16,-0.166)}%
 \rput[c](0,-3){\pspolygon[fillstyle=solid, fillcolor=lightgray](0.166,1.166)(0.825,0.825)(1.166,0.166)(-0.166,-1.166)(-0.825,-0.825)(-1.16,-0.166)}%
 \rput[c](3,3){\pspolygon[fillstyle=solid, fillcolor=lightgray](0.166,1.166)(0.825,0.825)(1.166,0.166)(-0.166,-1.166)(-0.825,-0.825)(-1.16,-0.166)}%
 \rput[c](4,2){\pspolygon[fillstyle=solid, fillcolor=lightgray](0.166,1.166)(0.825,0.825)(1.166,0.166)(-0.166,-1.166)(-0.825,-0.825)(-1.16,-0.166)}%
 \rput[c](5,1){\pspolygon[fillstyle=solid, fillcolor=lightgray](0.166,1.166)(0.825,0.825)(1.166,0.166)(-0.166,-1.166)(-0.825,-0.825)(-1.16,-0.166)}%
 \rput[c](-3,-3){\pspolygon[fillstyle=solid, fillcolor=lightgray](0.166,1.166)(0.825,0.825)(1.166,0.166)(-0.166,-1.166)(-0.825,-0.825)(-1.16,-0.166)}%
 \rput[c](-5,-1){\pspolygon[fillstyle=solid, fillcolor=lightgray](0.166,1.166)(0.825,0.825)(1.166,0.166)(-0.166,-1.166)(-0.825,-0.825)(-1.16,-0.166)}%
 \rput[c](-4,-2){\pspolygon[fillstyle=solid, fillcolor=lightgray](0.166,1.166)(0.825,0.825)(1.166,0.166)(-0.166,-1.166)(-0.825,-0.825)(-1.16,-0.166)}%
 \rput[c](0,3){\pspolygon[fillstyle=solid, fillcolor=lightgray](0.166,1.166)(0.825,0.825)(1.166,0.166)(-0.166,-1.166)(-0.825,-0.825)(-1.16,-0.166)}%
 \rput[c](-5,2){\pspolygon[fillstyle=solid, fillcolor=lightgray](0.166,1.166)(0.825,0.825)(1.166,0.166)(-0.166,-1.166)(-0.825,-0.825)(-1.16,-0.166)}%
   \rput[c](-3,3){\pspolygon[fillstyle=solid, fillcolor=lightgray](0.166,1.166)(0.825,0.825)(1.166,0.166)(-0.166,-1.166)(-0.825,-0.825)(-1.16,-0.166)}%
  \rput[c](-4,1){\pspolygon[fillstyle=solid, fillcolor=lightgray](0.166,1.166)(0.825,0.825)(1.166,0.166)(-0.166,-1.166)(-0.825,-0.825)(-1.16,-0.166)}%
  \rput[c](-2,2){\pspolygon[fillstyle=solid, fillcolor=lightgray](0.166,1.166)(0.825,0.825)(1.166,0.166)(-0.166,-1.166)(-0.825,-0.825)(-1.16,-0.166)}%
  \rput[c](2,1){\pspolygon[fillstyle=solid, fillcolor=lightgray](0.166,1.166)(0.825,0.825)(1.166,0.166)(-0.166,-1.166)(-0.825,-0.825)(-1.16,-0.166)}%
  \rput[c](-1,1){\pspolygon[fillstyle=solid, fillcolor=lightgray](0.166,1.166)(0.825,0.825)(1.166,0.166)(-0.166,-1.166)(-0.825,-0.825)(-1.16,-0.166)}
  \rput[c](1,2){\pspolygon[fillstyle=solid, fillcolor=lightgray](0.166,1.166)(0.825,0.825)(1.166,0.166)(-0.166,-1.166)(-0.825,-0.825)(-1.16,-0.166)}%
  \rput[c](3,0){\pspolygon[fillstyle=solid, fillcolor=lightgray](0.166,1.166)(0.825,0.825)(1.166,0.166)(-0.166,-1.166)(-0.825,-0.825)(-1.16,-0.166)}%
  \rput[c](1,-1){\pspolygon[fillstyle=solid, fillcolor=lightgray](0.166,1.166)(0.825,0.825)(1.166,0.166)(-0.166,-1.166)(-0.825,-0.825)(-1.16,-0.166)}%
  \rput[c](-1,-2){\pspolygon[fillstyle=solid, fillcolor=lightgray](0.166,1.166)(0.825,0.825)(1.166,0.166)(-0.166,-1.166)(-0.825,-0.825)(-1.16,-0.166)}%
  \rput[c](-2,-1){\pspolygon[fillstyle=solid, fillcolor=lightgray](0.166,1.166)(0.825,0.825)(1.166,0.166)(-0.166,-1.166)(-0.825,-0.825)(-1.16,-0.166)}%
  \rput[c](-3,0){\pspolygon[fillstyle=solid, fillcolor=lightgray](0.166,1.166)(0.825,0.825)(1.166,0.166)(-0.166,-1.166)(-0.825,-0.825)(-1.16,-0.166)}%
  \pspolygon[fillstyle=solid, fillcolor=gray](0.166,1.166)(0.825,0.825)(1.166,0.166)(-0.166,-1.166)(-0.825,-0.825)(-1.16,-0.166)%
  \rput[c](-4,4){\pspolygon[fillstyle=solid, fillcolor=lightgray](0.166,1.166)(0.825,0.825)(1.166,0.166)(-0.166,-1.166)(-0.825,-0.825)(-1.16,-0.166)}%
  \rput[c](-1,4){\pspolygon[fillstyle=solid, fillcolor=lightgray](0.166,1.166)(0.825,0.825)(1.166,0.166)(-0.166,-1.166)(-0.825,-0.825)(-1.16,-0.166)}%
  \rput[c](2,4){\pspolygon[fillstyle=solid, fillcolor=lightgray](0.166,1.166)(0.825,0.825)(1.166,0.166)(-0.166,-1.166)(-0.825,-0.825)(-1.16,-0.166)}%
  \psdots[dotsize=5pt](0,0)(2,1)(-1,1)(1,2)(3,0)(1,-1)(-1,-2)(-2,-1)(-3,0)(-2,2)(-3,3)(-4,1)(-5,2)(0,3)(-3,-3)(-5,-1)(-4,-2)(3,3)(4,2)(0,-3)(2,-2)(4,-1)(3,-3)(5,-2)%

  \cnode*(0,0){2pt}{origin}\nput[labelsep=2pt]{135}{origin}{$\mathbf{0}$}%
  \rput[c](-0.2,-0.6){$\pazocal{V}$}
\end{pspicture*} \hspace{0.5cm}
\begin{pspicture*}(-4,-2.7)(4,3.4)
  \psset{linecolor=gray}
  \pspolygon[fillstyle=solid,fillcolor=lightgray,linestyle=none](-4,-2.7)(-4,3.5)(4,3.5)(4,-2.7)
 \rput[c](5,-2){\pspolygon[fillstyle=solid, fillcolor=lightgray](0.166,1.166)(0.825,0.825)(1.166,0.166)(-0.166,-1.166)(-0.825,-0.825)(-1.16,-0.166)}
 \rput[c](3,-3){\pspolygon[fillstyle=solid, fillcolor=lightgray](0.166,1.166)(0.825,0.825)(1.166,0.166)(-0.166,-1.166)(-0.825,-0.825)(-1.16,-0.166)}
 \rput[c](4,-1){\pspolygon[fillstyle=solid, fillcolor=lightgray](0.166,1.166)(0.825,0.825)(1.166,0.166)(-0.166,-1.166)(-0.825,-0.825)(-1.16,-0.166)}
 \rput[c](2,-2){\pspolygon[fillstyle=solid, fillcolor=lightgray](0.166,1.166)(0.825,0.825)(1.166,0.166)(-0.166,-1.166)(-0.825,-0.825)(-1.16,-0.166)}
 \rput[c](0,-3){\pspolygon[fillstyle=solid, fillcolor=lightgray](0.166,1.166)(0.825,0.825)(1.166,0.166)(-0.166,-1.166)(-0.825,-0.825)(-1.16,-0.166)}
 \rput[c](3,3){\pspolygon[fillstyle=solid, fillcolor=lightgray](0.166,1.166)(0.825,0.825)(1.166,0.166)(-0.166,-1.166)(-0.825,-0.825)(-1.16,-0.166)}
 \rput[c](4,2){\pspolygon[fillstyle=solid, fillcolor=lightgray](0.166,1.166)(0.825,0.825)(1.166,0.166)(-0.166,-1.166)(-0.825,-0.825)(-1.16,-0.166)}
 \rput[c](5,1){\pspolygon[fillstyle=solid, fillcolor=lightgray](0.166,1.166)(0.825,0.825)(1.166,0.166)(-0.166,-1.166)(-0.825,-0.825)(-1.16,-0.166)}
 \rput[c](-3,-3){\pspolygon[fillstyle=solid, fillcolor=lightgray](0.166,1.166)(0.825,0.825)(1.166,0.166)(-0.166,-1.166)(-0.825,-0.825)(-1.16,-0.166)}
 \rput[c](-5,-1){\pspolygon[fillstyle=solid, fillcolor=lightgray](0.166,1.166)(0.825,0.825)(1.166,0.166)(-0.166,-1.166)(-0.825,-0.825)(-1.16,-0.166)}
 \rput[c](-4,-2){\pspolygon[fillstyle=solid, fillcolor=lightgray](0.166,1.166)(0.825,0.825)(1.166,0.166)(-0.166,-1.166)(-0.825,-0.825)(-1.16,-0.166)}
 \rput[c](0,3){\pspolygon[fillstyle=solid, fillcolor=lightgray](0.166,1.166)(0.825,0.825)(1.166,0.166)(-0.166,-1.166)(-0.825,-0.825)(-1.16,-0.166)}
 \rput[c](-5,2){\pspolygon[fillstyle=solid, fillcolor=lightgray](0.166,1.166)(0.825,0.825)(1.166,0.166)(-0.166,-1.166)(-0.825,-0.825)(-1.16,-0.166)}
   \rput[c](-3,3){\pspolygon[fillstyle=solid, fillcolor=lightgray](0.166,1.166)(0.825,0.825)(1.166,0.166)(-0.166,-1.166)(-0.825,-0.825)(-1.16,-0.166)}
  \rput[c](-4,1){\pspolygon[fillstyle=solid, fillcolor=lightgray](0.166,1.166)(0.825,0.825)(1.166,0.166)(-0.166,-1.166)(-0.825,-0.825)(-1.16,-0.166)}
  \rput[c](-2,2){\pspolygon[fillstyle=solid, fillcolor=lightgray](0.166,1.166)(0.825,0.825)(1.166,0.166)(-0.166,-1.166)(-0.825,-0.825)(-1.16,-0.166)}
  \rput[c](2,1){\pspolygon[fillstyle=solid, fillcolor=lightgray](0.166,1.166)(0.825,0.825)(1.166,0.166)(-0.166,-1.166)(-0.825,-0.825)(-1.16,-0.166)}
  \rput[c](-1,1){\pspolygon[fillstyle=solid, fillcolor=lightgray](0.166,1.166)(0.825,0.825)(1.166,0.166)(-0.166,-1.166)(-0.825,-0.825)(-1.16,-0.166)}
  \rput[c](1,2){\pspolygon[fillstyle=solid, fillcolor=lightgray](0.166,1.166)(0.825,0.825)(1.166,0.166)(-0.166,-1.166)(-0.825,-0.825)(-1.16,-0.166)}
  \rput[c](3,0){\pspolygon[fillstyle=solid, fillcolor=lightgray](0.166,1.166)(0.825,0.825)(1.166,0.166)(-0.166,-1.166)(-0.825,-0.825)(-1.16,-0.166)}
  \rput[c](1,-1){\pspolygon[fillstyle=solid, fillcolor=lightgray](0.166,1.166)(0.825,0.825)(1.166,0.166)(-0.166,-1.166)(-0.825,-0.825)(-1.16,-0.166)}
  \rput[c](-1,-2){\pspolygon[fillstyle=solid, fillcolor=lightgray](0.166,1.166)(0.825,0.825)(1.166,0.166)(-0.166,-1.166)(-0.825,-0.825)(-1.16,-0.166)}
  \rput[c](-2,-1){\pspolygon[fillstyle=solid, fillcolor=lightgray](0.166,1.166)(0.825,0.825)(1.166,0.166)(-0.166,-1.166)(-0.825,-0.825)(-1.16,-0.166)}
  \rput[c](-3,0){\pspolygon[fillstyle=solid, fillcolor=lightgray](0.166,1.166)(0.825,0.825)(1.166,0.166)(-0.166,-1.166)(-0.825,-0.825)(-1.16,-0.166)}
  \pspolygon[fillstyle=solid, fillcolor=lightgray](0.166,1.166)(0.825,0.825)(1.166,0.166)(-0.166,-1.166)(-0.825,-0.825)(-1.16,-0.166)
  \psset{linecolor=black}
    \pscircle[linestyle=solid,linecolor=darkgray,fillstyle=none,fillcolor=gray,opacity=0.5](-0.15,0.3){2.6}
  \multido{\N=-3+3}{4}{\rput[c](\N,0){\psline[linestyle=dotted](-4,4)(0,0)(3,-3)}}
  \multido{\N=-4.5+1.5}{7}{\rput[c](\N,0){\psline[linestyle=dotted](2,4)(0,0)(-2,-4)}}
  \multido{\N=-12+3}{10}{\rput[c](\N,0){\psline[linestyle=dotted](-4,-2)(0,0)(4,2)(8,4)}}
  \psdots[dotsize=5pt](0,0)(2,1)(-1,1)(1,2)(3,0)(1,-1)(-1,-2)(-2,-1)(-3,0)(-2,2)(-3,3)(-4,1)(-5,2)(0,3)(-3,-3)(-5,-1)(-4,-2)(3,3)(4,2)(0,-3)(2,-2)(4,-1)(3,-3)(5,-2)
  \psline(0,0)(2,1)
  \psline(0,0)(1,2)
  \psline(0,0)(1,-1)
  \psline(0,0)(-1,1)
  \psline(-1,1)(-2,2)
  \psline(-1,1)(1,2)
  \psline(-1,1)(-2,-1)
  \psline(0,0)(-2,-)
  \psline(0,0)(-1,-2)
  \psline(-1,-2)(1,-1)
  \psline(1,-1)(2,1)
  \psline(2,1)(1,2)
  \psline(-2,-1)(-1,-2)
  \cnode[fillstyle=solid,fillcolor=white](-0.15,0.3){3pt}{t}\nput[labelsep=2pt]{90}{t}{$t$}
  \SpecialCoor
  \rput[c](t){\pnode(3.2;45){A}} \rput[c](t){\pnode(2.6;45){B}}
  \ncline{->}{A}{B}
  \nput[labelsep=2pt]{0}{A}{$t+B_2^n$}
\end{pspicture*}
\caption{Voronoi cell $\pazocal{V}$ on the left. Graph $H = (\Lambda,F)$ depicted with dotted lines on the right. The edges in the induced subgraph $H[S]$ are drawn solid.}
\end{center}
\end{figure}
Then it follows from the work of \cite{CVP-Voronoi-Algo-MicciancioVoulgaris-SICOMP2013} that the subgraph induced by $\Lambda \cap (t + B_2^n)$ is connected. Hence, one can compute
the closest lattice point to $t$ (again using \cite{CVP-Voronoi-Algo-MicciancioVoulgaris-SICOMP2013}) and then
traverse the subgraph. 
\end{proof}
Next, we need to enumerate lattice points in arbitrary convex bodies.
For convex bodies $A,B \subseteq \setR^n$, the \emph{covering number} $N(A,B) := \min\{ N \mid \exists x_1,\dots,x_N \in \setR^n: A \subseteq \bigcup_{i=1}^N (x_i + B)\}$ is the minimum number of translates of $B$ needed to cover $A$.
  For a convex body $K \subseteq \setR^n$ and a full rank lattice $\Lambda \subseteq \setR^n$ we define
\[
 G(\Lambda,K) := \max_{x \in \setR^n} |(K + x) \cap \Lambda|.
\]
In words, $G(\Lambda,K)$ denotes the maximum number of lattice points that any shift of $K$
contains. Note that even if $K \cap \Lambda = \emptyset$, $G(\Lambda,K)$ might still be arbitrarily large.
However, algorithmically the quantity $G(\Lambda,K)$ is useful:
\begin{theorem}[\cite{EnumerateLatticeAlgosDadushPeikertVempalaFOCS11,NearOptDetAlgoForMEllipsoid-DadushVempalaPNAS13}  and {\cite[Theorem 5.2.6]{DadushThesis2012}}\label{thm:LatticeEnumeration}]
  Given a convex body $K \subseteq r B_2^n$ and the basis $B \in \setQ^{n \times n}$ of a full rank lattice $\Lambda := \Lambda(B) \subseteq \setR^n$,
  one can enumerate all points in $K \cap \Lambda$ in deterministic time $2^{O(n)} \cdot G(\Lambda,K)$
  times a polynomial in $r$ and the encoding length of $B$.
\end{theorem}
Again, we briefly sketch the algorithm behind Theorem~\ref{thm:LatticeEnumeration}:
\begin{proof}[Proof sketch]
We use the method of Dadush and Vempala~\cite{NearOptDetAlgoForMEllipsoid-DadushVempalaPNAS13} to compute an \emph{$M$-ellipsoid} $\pazocal{E}$ of $K$ which has the property that $N(K,\pazocal{E}),N(\pazocal{E},K) \leq 2^{O(n)}$.
Their deterministic algorithm takes time $2^{O(n)}$. In particular this means that $2^{-\Theta(n)} \leq \frac{G(\Lambda,K)}{G(\Lambda,\pazocal{E})} \leq 2^{\Theta(n)}$. 
Next, we compute\footnote{At least in the case that $\pazocal{E}$ is an $M$-ellipsoid for $K$, one may find
  those translates with $N \leq 2^{O(n)}N(K,\pazocal{E})$ with ease.
  After applying a linear transformation, we may assume that $\pazocal{E} = \sqrt{n} B_2^n$. Then take all translates $x + \pazocal{E}$ with $x \in \setZ^n$ that intersect $K$.} the translates $x_1,\dots,x_N$ with $N \leq 2^{O(n)}$ so that $K \subseteq \bigcup_{i=1}^N (x_i + \pazocal{E})$.
Then we use the algorithm from Theorem~\ref{thm:EnumerateLatticePointsInBall} 
to enumerate all lattice points in $(x_i + \pazocal{E}) \cap \Lambda$ for all $i=1,\dots,N$. We keep the ones that are in $K$ and terminate.
\end{proof}

Next, we require an upper bound on $G(\Lambda,K)$ in terms of the volume of $K$ and density of $\Lambda$.
Surprisingly, such an upper bound exists if we additionally control the covering radius.
We reproduce Dadush's proof as the argument is key to understanding the algorithm:
\begin{lemma}[{\cite[Lemma 7.4.1]{DadushThesis2012}}] \label{lem:UpperBoundOnLatticePoints}
  For any full rank lattice $\Lambda \subseteq \setR^n$ and any convex body $K \subseteq \setR^n$ one has
  \[
  G(\Lambda,K) \leq 2^n \max\{ \mu(\Lambda,K)^n,1\} \cdot \frac{\Vol_n(K)}{\det(\Lambda)}.
  \]
\end{lemma}

\begin{proof}
  After a linear transformation and scaling by $\max\{ \mu(\Lambda,K),1\}$, the statement is
  equivalent to the following simpler claim: \\
  {\bf Claim.} \emph{For any convex body $K \subseteq \setR^n$ with $\mu(\setZ^n,K) \leq 1$ and any $z \in \setR^n$ one has $|K \cap (z+\setZ^n)| \leq 2^n \Vol_n(K)$.} \\
{\bf Proof of Claim.}
The claim is invariant under translating $K$, hence we may assume that $z=\bm{0} \in K$.
Let $\equiv$ be the equivalence relation on pairs $x,y \in K$ that is defined by  $x \equiv y \Leftrightarrow x-y \in \setZ^n$. We define a set $V \subseteq K$ by selecting one element from each equivalence class w.r.t. $\equiv$. It would not
matter much which element was selected, but let us make the canonical choice of choosing the lexicographically
minimal one. In other words, we choose 
  \[
   V = \big\{ x \in K \mid x \leq_{\textrm{lex}} y \quad \forall y \in  K \cap (x + \setZ^n)\big\},
  \]
  where $\leq_{\textrm{lex}}$ is the standard lexicographical ordering.
 \iftrue
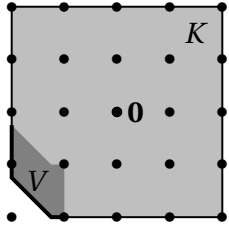
  \begin{center}
    \psset{unit=0.7cm}
    \begin{pspicture}(-2,-2)(2,2)
      \pspolygon[fillstyle=solid,fillcolor=lightgray,linewidth=0.75pt](-1.25,-2)(-2,-1.25)(-2,2)(2,2)(2,-2)\rput[c](1.5,1.5){$K$}
      \pspolygon[fillstyle=solid,fillcolor=gray,linestyle=none](-1,-2)(-1.25,-2)(-2,-1.25)(-2,-0.25)(-1.25,-1)(-1,-1)
      \psline[linewidth=1.5pt](-1,-2)(-1.25,-2)(-2,-1.25)(-2,-0.25)\rput[c](-1.5,-1.3){$V$}
      \multido{\N=-2+1}{5}{\multido{\n=-2+1}{5}{\psdots(\N,\n)}}
      \cnode*(0,0){2.0pt}{origin}\nput[labelsep=2pt]{0}{origin}{$\bm{0}$}
    \end{pspicture}
  \end{center}
 \fi
As we select at most one element from
  each equivalence class, we certainly have $\textrm{Vol}_n(V) \leq 1$. On the other hand,
  $\mu(\setZ^n,K) \leq 1$ implies that for all $x \in \setR^n$ one has  $(x+\setZ^n) \cap K \neq \emptyset$.
  That in turn means that every equivalence class has a member in $K$ and so $\Vol_n(V) \geq 1$.
  Together this gives $\textrm{Vol}_n(V)=1$.
  Next, we note that by construction all translates $x+V$ with $x \in \setZ^n$ are disjoint. Moreover, for $x \in  K \cap \setZ^n$ one has that $x + V \subseteq K + K =2K$. Then 
  \[
  |K \cap \setZ^n| = \sum_{x \in K \cap \setZ^n} \underbrace{\Vol_n(x+V)}_{=1} \stackrel{\textrm{disj.}}{=} \Vol_n\Big(\bigcup_{x \in K \cap \setZ^n} (x+V)\Big) \leq \Vol_n(2K),
  \]
which gives the claim.
\end{proof}


One technicality we have to deal with is that Theorem~\ref{thm:FindingSubspaceIn2ToN} requires a lower bound
on the \emph{inradius} of $K$. Hence we run a preprocessing step: if there is no suitable lower bound for the
inradius, then the lattice points of $K$ are all contained in an easy-to-find hyperplane.
\begin{lemma} \label{lem:WellScaleKorDecideKisThin}
  Given a compact convex set $K \subseteq r B_2^n$ and a lattice $\Lambda = \Lambda(B)$. Then in time polynomial in $n$,
  times a polynomial in $\log(r)$ and the encoding length of $B$ one can find at least one of the following: 
  \begin{enumerate}
  \item[(a)] An ellipsoid $\pazocal{E}$ and center $c$ so that $c + \frac{1}{(n+1)^{3/2}}\pazocal{E} \subseteq K \subseteq c+\pazocal{E}$.
  \item[(b)] A vector $a \in \setR^n \setminus \{ \bm{0}\}$ and $\beta \in \setR$ so that $K \cap \Lambda \subseteq \{ x \in \setR^n \mid \left<a,x\right> = \beta\}$.
  \end{enumerate}
\end{lemma}
\begin{proof}
  We may assume that $\textrm{rank}(\Lambda)=n$, otherwise any $a$ orthogonal to $\textrm{span}(\Lambda)$ will satisfy (b).
  Next, we use a variant of the ellipsoid method from \cite{Groetschel1988} (see also Lemma~2.5.10 in \cite{DadushThesis2012}) to find a pair  $(c,\pazocal{E})$ in time polynomial in $n$, $\log(r)$ and $\log(\frac{1}{\varepsilon})$ 
  so that either (a) holds, or   $K \subseteq c+\pazocal{E}$ and $\Vol_n(\pazocal{E}) \leq \varepsilon$. 
   Suppose the latter happens. Then using Minkowski's Theorem (Theorem~\ref{thm:Minkowski}) in $(*)$ and the Blaschke-Santal\'o-Bourgain-Milman Theorem (Theorem~\ref{thm:BSBM}) in $(**)$ we obtain
   \begin{eqnarray*}
     \lambda_1(\Lambda^*,\pazocal{E}^{\circ}) \stackrel{(*)}{\lesssim} \Big(\frac{\det(\Lambda^*)}{\Vol_n(\pazocal{E}^{\circ})}\Big)^{1/n}  \stackrel{(**)}{\lesssim}  \Big(\frac{\Vol_n(\pazocal{E})}{\det(\Lambda) \cdot \Vol_n(B_2^n)^2}\Big)^{1/n}
     \lesssim n \cdot \Big(\frac{\varepsilon}{\det(\Lambda)}\Big)^{1/n} \leq \frac{1}{4} \cdot 2^{-n/2},
   \end{eqnarray*}
   for a suitable choice of $\varepsilon>0$. Then the LLL-algorithm (see Theorem~\ref{thm:LLLalgorithm}) can find a dual lattice vector $a \in \Lambda^* \setminus \{ \bm{0}\}$ with $\|a\|_{\pazocal{E}^{\circ}} \leq 2^{n/2} \cdot \lambda_1(\Lambda^*,\pazocal{E}^{\circ}) \leq \frac{1}{4}$. That vector $a$ with $\beta := \lceil \left<a,c\right> \rfloor$ will  satisfy (b). 
   To see this, we note that for any $x \in \Lambda$ we have $\left<a,x\right> \in \setZ$ as $a \in \Lambda^*$.
   Moreover, for $x \in K$ we have
   \[
     |\left<a,x\right>-\beta| \leq |\left<a,x-c\right>|+\frac{1}{2} \leq \underbrace{\|a\|_{E^{\circ}}}_{\leq 1/4}\underbrace{\|x-c\|_{E}}_{\leq 1} + \frac{1}{2} \leq \frac{3}{4}
   \]
   using Cauchy-Schwarz. 
\end{proof}
We are now ready to state the complete algorithm. As mentioned earlier, we denote $\rho(n) := \Theta(\log^4(2n))$ 
as the approximation factor from Theorem~\ref{thm:FindingSubspaceIn2ToN}.

\begin{center}
  \psframebox{\begin{minipage}{14cm}
{\sc Dadush's algorithm} \cite[Algorithm 7.2]{DadushThesis2012} \\
{\bf Input:} Compact convex set $K \subseteq \setR^n$, lattice $\Lambda \subseteq \setR^n$ \\
{\bf Output:} Point $x \in K \cap \Lambda$ or decision that there is none
\begin{enumerate*}
\item[(1)] If $n = 1$, use binary search to find integer multiple of $\lambda_1 (\Lambda, [-1,1])$ in $K$ or certify none exists.
\item[(2)] Use Lemma~\ref{lem:WellScaleKorDecideKisThin}. If case (b) happens, obtain hyperplane $H$ with $K \cap \Lambda \subseteq H$. Recurse on $\textsc{Dadush}(K \cap H, \Lambda \cap H)$ and return the answer.
\item[(3)] Compute a subspace $W \subseteq \setR^n$ with $d := \dim(W)$ and $R := R_{W}(\Lambda,K)$ so that $R \leq \mu(\Lambda,K) \leq \rho(n) \cdot R$.
\item[(4)] Set $\tilde{K} := \min\{ \rho(n) \cdot R,1 \} \cdot (K-c)+c$ for some $c \in K$.
\item[(5)] Compute an $M$-ellipsoid $\pazocal{E} \subseteq W$ for $\Pi_W(\tilde{K})$.
\item[(6)] Compute $N \leq 2^{O(d)}$ points $x_1,\dots,x_N \in W$ so that $\Pi_W(\tilde{K}) \subseteq \bigcup_{i=1}^N (x_i + \pazocal{E})$. 
\item[(7)] Compute $X := \Pi_W(\tilde{K}) \cap \Pi_W(\Lambda) = \big(\bigcup_{i=1}^N ((x_i+\pazocal{E}) \cap \Pi_W(\Lambda))\big) \cap \Pi_W(\tilde{K})$.
\item[(8)] Recursively call ${\textsc{Dadush}}(\tilde{K} \cap \Pi_{W}^{-1}(x),\Lambda \cap \Pi_W^{-1}(x))$ for all $x \in X$ and return any found lattice point (if there is any).
\end{enumerate*}
\end{minipage}}
\end{center}
Here, to be more informative, we have expanded the blackbox from Theorem~\ref{thm:LatticeEnumeration}
into lines (5)-(7).
The reader may also note a subtlety here that we have not discussed so far: if $K$ is very large
so that $\mu(\Lambda,K) \ll 1$, then we may shrink $K$ to a smaller body $\tilde{K} \subseteq K$
as long as we ensure that still $\mu(\Lambda, \tilde{K}) \leq 1$. We can now finish the analysis:
\begin{theorem}
  Let $K \subseteq r B_2^n$ be a closed convex set and let $\Lambda := \Lambda(B)$ be a full
  rank lattice given by its basis $B \in \setQ^{n \times n}$.
  Then Dadush's algorithm finds a point in $K \cap \Lambda$ in time $(\log(2n))^{O(n)}$ times a polynomial in $\log(r)$
  and the encoding length of $B$,  if there is such a point.
\end{theorem}

\begin{proof}
  If the algorithm recurses in (2), the claim is clear by induction. So assume otherwise
  which in particular means that $K$ is full-dimensional. 
  First we argue correctness of the algorithm.
  Let  $s := \min\{ \rho(n) \cdot R,1\} \in [0,1]$  and recall that $\tilde{K} \subseteq K$ is a scaling of $K$
  by a factor of $s$. After step (4), the algorithm searches for a lattice point in $\tilde{K}$ rather than in the
  original body $K$. If $s<1$, then the covering radius of the shrunk body is $\mu(\Lambda,\tilde{K}) = \frac{1}{\rho(n) \cdot R} \mu(\Lambda,K) \leq 1$. In other words, even though we continue the search in the strictly smaller body $\tilde{K}$, we are still guaranteed that $\tilde{K} \cap \Lambda \neq \emptyset$. 
 Next, we discuss the running time of the algorithm. We estimate that 
  \begin{eqnarray*}
  G(\Pi_W(\Lambda),\Pi_W(\tilde{K})) &\stackrel{\textrm{Lem~\ref{lem:UpperBoundOnLatticePoints}}}{\leq}& 2^d \max\big\{ \mu(\Pi_W(\Lambda),\Pi_W(\tilde{K}))^d,1\big\} \cdot \frac{\Vol_d(\Pi_W(\tilde{K}))}{\det(\Pi_W(\Lambda))} \\
                                          &\leq& 2^d \max\Big\{ \Big(\underbrace{\frac{\rho(n) R}{s}}_{\geq 1}\Big)^d,1\Big\} \cdot s^d \cdot \underbrace{\frac{\Vol_d(\Pi_W(K))}{\det(\Pi_W(\Lambda))}}_{=R^{-d}} \\
    &=& 2^d \cdot (\rho(n) R)^d \cdot R^{-d} = (2 \rho(n))^d.
  \end{eqnarray*}
  Here we use that $\mu(\Pi_W(\Lambda),\Pi_W(\tilde{K})) \leq \mu(\Lambda,\tilde{K}) = \frac{1}{s} \cdot \mu(\Lambda,K) \leq \frac{\rho(n) \cdot R}{s}$. Then $|X| \leq G(\Pi_W(\Lambda),\Pi_W(\tilde{K})) \leq 2^d\rho(n)^d$ by
  Lemma~\ref{lem:UpperBoundOnLatticePoints}. 
Now, let $T(n)$ be the maximum number of recursive calls of the algorithm on $n$-dimensional instances. Then we have the recursion
\[
 T(n) \leq \max_{d \in \{ 1,\dots,n\}} \Big\{ 1 + (O(1) \cdot \rho(n))^d \cdot T(n-d)\Big\} \quad \textrm{and} \quad T(0)=0,
\]
which indeed resolves to  $T(n) \leq O(\rho(n))^n$.

It remains to analyze the running time per iteration.
Let $\Lambda$ and $K \subseteq r B_2^n$ be the original input instance. After applying a linear
transformation (and adjusting the radius $r$) we may assume that $\Lambda = \setZ^n$. Now consider
an arbitrary instance $(\Lambda_0,K_0)$ that appears in the recursion tree.
The lattice $\Lambda_0$ and the convex set $K_0$ arise by iteratively intersecting
$\Lambda$ and $K$ with affine subspaces, i.e. $\Lambda_0 = \setZ^n \cap H$ and $K_0 = K \cap H$ for some
affine subspace $H$ and so $K_0 \subseteq r B_2^n$.
But how do we know that the encoding length of $H$ did not blow up too much in the recursive process? By a slight abuse of notation, we say that a matrix $A \in \setR^{n \times (d+1)}$ is a \emph{basis} of a $d$-dimensional affine subspace $H \subseteq \setR^n$ if $H = A^0 + \textrm{span}(A^1,\dots,A^d)$ where $A^0,A^1,\dots,A^d$ are the columns of $A$. 
We make use of the following classical result on simultaneous diophantine approximation:
\begin{theorem}[Frank, Tardos~\cite{FrankTardos1987}] 
  Let $H \subseteq \setR^n$ be an affine subspace given by a basis $A \in \setQ^{n \times (d+1)}$ and let $N \in\setN$ be a parameter. Then in time polynomial in $n$, $\log(N)$ and the encoding length of $A$
  one can compute an affine subspace $H' \subseteq \setR^n$ so that $H \cap \{ -N,\dots,N\}^n = H' \cap \{ -N,\dots,N\}^n$. Moreover the encoding length of the basis of $H'$ is polynomial in $n$ and $\log(N)$ (in particular that bound is independent of the encoding length of $A$). 
\end{theorem}
That means we can in each iteration replace the current affine subspace $H$ by  $H'$ and
keep the bit complexity polynomial in $n$ and $\log(r)$ in all iterations. Once we have this insight,
the remainder of the analysis is simple. If $W_0$ is the subspace computed for pair $(\Lambda_0,K_0)$,
then $X_0 := \Pi_{W_0}(\tilde{K}_0) \cap \Pi_{W_0}(\Lambda_0)$ can be computed in time $(2\rho(n))^{\dim(W_0)}$
times a polynomial in the encoding length of $\Lambda_0$ and $K_0$ which in turn is bounded by a polynomial in $n$, $\log(r)$ and the encoding length of the original lattice basis $B$. 
\end{proof}

We also explain how Dadush's algorithm can be used to solve integer linear programs in time $(\log(2 n))^{O(n)}$.
Again, the arguments used are standard. 
\begin{proof}[Proof of Corollary~\ref{cor:SolvingExplicitIPinLogNtoN}]
  Consider an arbitrary integer linear program $\max\{ c^Tx \mid Ax \leq b, x \in \setZ^n\}$. One can
  compute a number $M$ in time polynomial in $n$ and the encoding length of $A$ and $b$ so that
  if the IP is bounded and feasible, then the optimum value is the same as $\max\{ c^Tx \mid Ax \leq b, \|x\|_{\infty} \leq M, x \in \setZ^n\}$ (see Schrijver~\cite[Cor 17.1c]{TheoryOfLPandILP-Schrijver1999}).
  Next, by applying binary search, it suffices to find an integer point in the compact convex
  set $K = \{ x \in \setR^n \mid c^Tx \geq \delta, Ax \leq b, \|x\|_{\infty} \leq M\}$ for which Theorem~\ref{thm:SolvingIPinLogNtoN} applies.
\end{proof}
As a side remark, Dadush~\cite[Chapter 7]{DadushThesis2012} refers to his algorithm as a \emph{Kannan-type algorithm} in reference to the algorithm of Kannan~\cite{n-to-n-algos-for-SVP-CVP-Kannan-MOR1987}
which also recurses on the fibers of a projection. This is in contrast to \emph{Lenstra-type algorithms} where one always recurses on slices with $(n-1)$-dimensional hyperplanes.

\section{Implications of Theorem~\ref{thm:KLConj}\label{sec:Implications}}

Here we derive a few implications of our main result. The following classical inequality will be useful here:

\begin{lemma}[\cite{Rogers1957TheDB}] \label{lem:RSineq}
For any convex set $K \subseteq \setR^n$ we have $\Vol_n (K-K) \le {2n \choose n} \cdot \Vol_n (K)$.
\end{lemma}
This implies that any volume-based lower bound for the covering radius is the same for $K$ and for $K-K$ up to constant factors. We note that this observation was already made by Dadush~\cite{Dadush-Finding-DenseLatticeSubspacesSTOC19}.
\begin{lemma} \label{lem:RvalueForKvsDifferenceBody}
  For any full rank lattice $\Lambda \subseteq \setR^n$, any convex body $K \subseteq \setR^n$ and any subspace $W \subseteq \setR^n$ one has
  $
 R_W(\Lambda,K-K) \leq R_W(\Lambda,K) \leq 4 R_W(\Lambda,K-K).
  $
\end{lemma}
\begin{proof}
  The first inequality is trivial as $K - K \supseteq K-x_0$ for any $x_0 \in K$. For the second inequality, let $d := \dim(W)$. Then
  \[
    R_W(\Lambda,K) = \Big(\frac{\det(\Pi_W(\Lambda))}{\Vol_d(\Pi_W(K))}\Big)^{1/d} \leq 4 \cdot \Big(\frac{\det(\Pi_W(\Lambda))}{\Vol_d(\Pi_W(K-K))}\Big)^{1/d} = 4 \cdot R_{W}(\Lambda,K-K)
    \]
  by the Rogers-Sheppard inequality (Lemma~\ref{lem:RSineq}) applied to $P := \Pi_W(K)$.
\end{proof}
We restate Theorem~\ref{thm:CoveringRadiusKvsKminusK}, which yields a nearly tight relationship between the covering radii of $K$ and $K-K$. We remark that it remains an open question whether the two quantities are equal up to a constant.

\begin{theorem*}[Theorem~\ref{thm:CoveringRadiusKvsKminusK}] 
For any full rank lattice $\Lambda \subseteq \setR^n$ and any convex body $K \subseteq \setR^n$, one has \[\mu(\Lambda,K-K) \leq \mu(\Lambda,K) \leq O(\log^{3} (2n)) \cdot \mu(\Lambda,K-K).\]
\end{theorem*}

\begin{proof}
Let $W$ denote the subspace attaining $\mu_{KL}(\Lambda, K)$. We can use Theorem~\ref{thm:KLConj} to upper bound 
 \begin{eqnarray*}
   \mu(\Lambda, K) \lesssim \log^{3} (2n) \cdot \mu_{KL} (\Lambda, K) & = & \log^{3} (2n) \cdot R_W(\Lambda,K) \\ & \stackrel{\textrm{Lem~\ref{lem:RvalueForKvsDifferenceBody}}}{\lesssim} & \log^{3} (2n) \cdot R_W(\Lambda,K-K) \\
                                                                      & \leq & \log^{3} (2n) \cdot \mu(\Lambda, K-K). \qedhere
 \end{eqnarray*}
\end{proof}

This has a well known immediate consequence on the \emph{flatness constant} in dimension $n$:
\begin{theorem} 
  For any convex body $K \subseteq \setR^n$ and any full rank lattice $\Lambda \subseteq \setR^n$, one has
  \[
   \mu(\Lambda,K) \cdot \lambda_1(\Lambda^{*}, (K-K)^{\circ}) \leq O(n \log^{4}(2n)).
  \]
\end{theorem}
\begin{proof}
Banaszczyk~\cite{Banaszczyk1996TransferenceTheoremsForGeneralConvexBodies} proved that for any symmetric convex
  body $Q \subseteq \setR^n$ one has $\mu(\Lambda,Q) \cdot \lambda_1(\Lambda^*,Q^{\circ}) \leq O(n \log(2n))$. 
  Setting $Q := K-K$ (which is a symmetric convex body) one then has by Theorem~\ref{thm:CoveringRadiusKvsKminusK}
  \[\mu(\Lambda,K) \cdot \lambda_1(\Lambda^{*}, Q^{\circ}) \leq O(\log^3(2n)) \cdot \mu(\Lambda,Q) \cdot \lambda_1(\Lambda^{*}, Q^{\circ}) \leq O(n\log^{4}(2n)). \qedhere\]
\end{proof}

Dadush~\cite{DadushImprovedFlatnessArgument} suggested an improvement to $O(n \log^3(2n))$ which we describe in the next two statements.

\begin{lemma} \label{lem:RtimesL1AtMostDimW}
  For any full rank lattice $\Lambda \subseteq \setR^n$ and any convex body $K \subseteq \setR^n$ and any subspace $W \subseteq \setR^n$ one has
\[
  R_W(\Lambda,K) \cdot \lambda_1(\Lambda^*,(K-K)^{\circ}) \leq O(\dim(W)).
\]
\end{lemma}
\begin{proof}
  Let $Q := K-K$ and abbreviate $d := \dim(W)$. Then
 \begin{eqnarray*}
  R_W(\Lambda,K) &\stackrel{\textrm{Lem~\ref{lem:RvalueForKvsDifferenceBody}}}{\lesssim}& R_W(\Lambda,Q) \\
&=& \Big(\frac{\det(\Pi_W(\Lambda))}{\Vol_d(\Pi_W(Q))}\Big)^{1/d} \\ & \stackrel{\textrm{Lem~\ref{thm:BSBM}}}{\asymp} & d \cdot \Big(\frac{\Vol_d(Q^\circ \cap W)}{\det(\Lambda^* \cap W)}\Big)^{1/d} \\ & \stackrel{\textrm{Thm~\ref{thm:Minkowski}}}{\lesssim} & d \cdot \frac{2}{\lambda_1(\Lambda^* \cap W, Q^\circ \cap W)} \lesssim \frac{d}{\lambda_1(\Lambda^*, Q^\circ)}.
 \end{eqnarray*}
Here, we have used that $\Pi_W(\Lambda)^* = \Lambda^* \cap W$.
Moreover, we used $\lambda_1(\Lambda^*, Q^\circ) \le \lambda_1(\Lambda^* \cap W, Q^\circ \cap W)$ in the last step. The claim follows.
\end{proof}
Now we can obtain a flatness bound by directly using our main result, Theorem~\ref{thm:KLConj}. 
\begin{theorem}
  For any full rank lattice $\Lambda \subseteq \setR^n$ and any convex body $K \subseteq \setR^n$ one has
  \[
    \mu(\Lambda,K) \cdot \lambda_1(\Lambda^*, (K-K)^{\circ}) \leq O(n \log^3(2n))
  \]
\end{theorem}
\begin{proof}
  Let $W$ be the subspace from Theorem~\ref{thm:KLConj} with $\mu(\Lambda,K) \leq O(\log^3 (2n)) \cdot R_W(\Lambda,K)$.
  Then apply Lemma~\ref{lem:RtimesL1AtMostDimW} with $\dim(W) \leq n$.
\end{proof}

However, one can shave off another logarithmic factor:

\begin{theorem}
  For any full rank lattice $\Lambda \subseteq \setR^n$ and any convex body $K \subseteq \setR^n$ one has
  $\mu(\Lambda,K) \cdot \lambda_1(\Lambda^*, (K-K)^{\circ}) \leq O(n \log^2(2n))$.
\end{theorem}
\begin{proof}
  Let $W_0,\dots,W_{k-1} \subseteq \setR^n$ be the subspaces from Cor~\ref{cor:KLtypeInequWithSequenceOfSubspaces} so that
  \[
    \mu(\Lambda,K) \leq O(\log^2 (2n)) \cdot \sum_{i=0}^{k-1} R_{W_i}(\Lambda,K).
  \]
  Recall the dimensions of those subspaces are geometrically decreasing, i.e. $\dim(W_i) \leq \frac{n}{2^i}$ for all $i \geq 0$. Then applying Lemma~\ref{lem:RtimesL1AtMostDimW} for each subspace $W_i$ gives
  \begin{eqnarray*}
    \mu(\Lambda,K) \cdot \lambda_1(\Lambda^*, (K-K)^{\circ})
&\lesssim& \log^2(2n) \cdot \sum_{i=0}^{k-1} R_{W_i}(\Lambda,K) \cdot \lambda_1(\Lambda^*, (K-K)^{\circ}) \\
&\lesssim& \log^2(2n) \cdot \sum_{i=0}^{k-1} \dim(W_i) \\ &\lesssim& n \log^2(2n). \qedhere
           \end{eqnarray*}
\end{proof}

We also explain the proof of Corollary~\ref{cor:FlatnessConstantSimple} which again is a known consequence (see e.g. \cite[Page 197]{DadushThesis2012}): 
\begin{corollary*}[Cor~\ref{cor:FlatnessConstantSimple}]
  Let $K \subseteq \setR^n$ by a convex body with $K \cap \setZ^n = \emptyset$.
  Then there is a vector $c \in \setZ^n \setminus \{ \bm{0}\}$ so that at most $O(n \log^{2}(2n))$ many hyperplanes of the form $\langle c, x \rangle = \delta$ with $\delta \in \setZ$ intersect $K$.
\end{corollary*}
\begin{proof}
  We apply Theorem~\ref{thm:FlatnessConstant} for the lattice $\Lambda := \setZ^n$ so that
  $\Lambda^* = \setZ^n$. Then $K \cap \setZ^n = \emptyset$ implies that $\mu(\setZ^n,K) > 1$
  and so $\lambda_1(\setZ^n, (K-K)^{\circ}) \lesssim n \log^2(2n)$. Let $c \in \setZ^n \setminus \{ \bm{0}\}$ be the vector attaining this bound. Then
  revisiting the definition of the dual norm (Sec~\ref{sec:LValueAndVolEstimates}) we have
  $
    \max\{ \left<c,x-y\right> : x,y \in K\} = \|c\|_{(K-K)^{\circ}} 
  $.
  That means at most $\|c\|_{(K-K)^{\circ}} +1 \lesssim n \log^2(2n)$ hyperplanes
  of the form $\left<c,x\right> = \delta$ with $\delta \in \setZ$ intersect $K$.
\end{proof}


\section{The approximate canonical filtration} \label{appendix:ApproximateFiltration}

In this section, we prove Theorem~\ref{thm:ApproxFilt}. The proof idea is rather simple: given a $t$-stable
filtration $\{ \bm{0}\} = \Lambda_0 \subset \dots \subset \Lambda_k = \Lambda$, we select one index from every density class in order to make the filtration well-separated. But before we come to the main argument, we require two lemmas.
\begin{lemma}[Grayson's parallelogram rule~\cite{Stuhler1976, Grayson1984}] \label{lem:ParallelogramRule} 
For any two lattices $\Lambda, \Lambda' \subseteq \setR^n$,
\[\det(\Lambda) \cdot \det (\Lambda') \ge \det(\Lambda + \Lambda') \cdot \det(\Lambda \cap \Lambda'). \]
\end{lemma}
A proof may also be found in Chapter 2 of \cite{PhDThesisStephens-Davidowitz2017}.
The $t$-stable filtration can be used to obtain lower bounds on the determinant of any sublattice: 

\begin{lemma}\label{lem:CanonicalFiltrationLowerBound}

Let $\Lambda \subseteq \setR^n$ be any lattice and let $\{ \bm{0}\} = \Lambda_0 \subset \Lambda_1 \subset \dots \subset \Lambda_k = \Lambda$ be a $t$-stable filtration.  Then for any sublattice $\{ \bm{0} \} \subset \tilde{\Lambda} \subseteq \Lambda$ we have
\[ \nd(\tilde{\Lambda}) \ge t^{-1} \cdot \nd(\Lambda_1). \]
\end{lemma}
\begin{proof} 
  Let $r_i := \nd(\Lambda_i / \Lambda_{i-1}) = \det(\Lambda_i / \Lambda_{i-1})^{1/\textrm{rank}(\Lambda_i/\Lambda_{i-1})}$ be the normalized determinant. We prove by induction on $i \in \{1, \dots, k\}$ that the result holds for all lattices $\tilde{\Lambda} \subseteq \Lambda_i$. The base case follows as $\Lambda_1 = \Lambda_1 / \Lambda_0$ is a scalar multiple of the $t$-stable
  lattice $\frac{1}{\textrm{nd}(\Lambda_1)}\Lambda_1$. Now suppose that $\tilde{\Lambda} \subseteq \Lambda_i$ and $\tilde{\Lambda} \not \subseteq \Lambda_{i-1}$ for some $i > 1$. Note that $\Lambda_+ := \tilde{\Lambda} + \Lambda_{i-1}$ satisfies $\Lambda_{i-1} \subset \Lambda_+ \subseteq \Lambda_{i}$, so that $\Lambda_+ / \Lambda_{i-1} \subseteq \Lambda_i/\Lambda_{i-1}$ and $\nd(\Lambda_+ /\Lambda_{i-1}) \ge t^{-1} \cdot r_i >t^{-1}\cdot r_1$. By Lemma~\ref{lem:ParallelogramRule},
  \[
    \det(\tilde{\Lambda}) \cdot\det(\Lambda_{i-1})  \ge \det(\tilde{\Lambda} + \Lambda_{i-1}) \cdot \det(\tilde{\Lambda} \cap \Lambda_{i-1}).
  \]
  Factoring out $\Lambda_{i-1}$ gives
  \[
   \det(\tilde{\Lambda}) \geq \det(\Lambda_+ / \Lambda_{i-1}) \cdot \det(\tilde{\Lambda} \cap \Lambda_{i-1}).
  \]

  If $\tilde{\Lambda} \cap \Lambda_{i-1} = \{\bm{0}\}$ then $ \det(\tilde{\Lambda}) \geq \det(\Lambda_+ / \Lambda_{i-1}) \cdot 1$ and so $\nd(\tilde{\Lambda}) \ge \nd(\Lambda_+ /\Lambda_{i-1}) > t^{-1}\cdot r_1$. Otherwise, $\{ \bm{0} \} \subset \tilde{\Lambda} \cap \Lambda_{i-1}$ and so
  \[
    \nd(\tilde{\Lambda}) \ge \nd(\Lambda_+/\Lambda_{i-1})^{\mathrm{rank}(\Lambda_+ / \Lambda_{i-1})/\mathrm{rank}(\tilde{\Lambda})} \cdot \nd(\tilde{\Lambda} \cap \Lambda_{i-1})^{\mathrm{rank}(\tilde{\Lambda} \cap \Lambda_{i-1})/\mathrm{rank}(\tilde{\Lambda})} \ge t^{-1} \cdot r_1
  \]
  where we used the inductive hypothesis on $\tilde{\Lambda} \cap \Lambda_{i-1} \subseteq \Lambda_{i-1}$ together with the fact that $\mathrm{rank}(\Lambda_+/\Lambda_{i-1}) + \mathrm{rank}(\tilde{\Lambda}\cap \Lambda_{i-1}) = \mathrm{rank}(\tilde{\Lambda})$. \end{proof}

Now, we come to the main argument:
\begin{proof}[Proof of Theorem~\ref{thm:ApproxFilt}]
  Let $r_i := \nd(\Lambda_i / \Lambda_{i-1})$ and $d_i := \textrm{rank}(\Lambda_i / \Lambda_{i-1})$. For $\ell \in \setZ$ denote $I_{\ell} := \{ i \in [k] : 2^{\ell} \leq r_i < 2 \cdot 2^{\ell}\}$.
  We define a sequence of indices $0=\ell(0)<\ell(1)<\dots<\ell(\tilde{k}) = k$ that contains precisely the largest index $i$ in each $I_{\ell}$ with $I_{\ell} \neq \emptyset$ plus the index $\ell(0) = 0$.
  We set  $\tilde{\Lambda}_j := \Lambda_{\ell(j)}$ and $\tilde{r}_j := \nd(\tilde{\Lambda}_j / \tilde{\Lambda}_{j-1})$. First, 
  consider an index $\ell$ with $I_{\ell} \neq \emptyset$. Let $i_{\min},i_{\max} \in I_{\ell}$ be the minimal and maximal indices in $I_\ell$. Then
  \begin{eqnarray*}
    \det(\Lambda_{i_{\max}} / \Lambda_{i_{\min}-1})^{1/\textrm{rank}(\Lambda_{i_{\max}}/\Lambda_{i_{\min}-1})} &=& \Big(\prod_{i=i_{\min}}^{i_{\max}} \det(\Lambda_{i}/\Lambda_{i-1})\Big)^{1 / \sum_{i=i_{\min}}^{i_{\max}} \textrm{rank}(\Lambda_i/\Lambda_{i-1})} \\
    &=& \Big(\prod_{i=i_{\min}}^{i_{\max}} r_i^{d_i}\Big)^{1/\sum_{i=i_{\min}}^{i_{\max}} d_i}.
  \end{eqnarray*}
  Note that this value is a weighted geometric average of $r_i$-values for $i \in I_{\ell}$.
  From this it immediately follows that $\tilde{r}_1 < \dots < \tilde{r}_{\tilde{k}}$ and $\tilde{r}_j \leq \frac{1}{2}\tilde{r}_{j+2}$ for all $j$.  
  It remains to show that the quotient lattices are scalar multiples of $2t$-stable lattices. Fix some index $j \in [\tilde{k}]$ and let $\Lambda' := \frac{1}{\tilde{r}_j} (\tilde{\Lambda}_{j}/\tilde{\Lambda}_{j-1})$. First note that by assumption, the filtration  $\{\bm{0}\} = \Lambda'_0 \subset \cdots \subset \Lambda'_{k'} := \Lambda'$ given by $\Lambda'_i := \frac{1}{\tilde{r}_j} (\Lambda_{\ell(j-1) + i}/\Lambda_{\ell(j-1)})$ with $k' := \ell(j) - \ell(j-1)$ is also $t$-stable
  because $\Lambda_{i+1}'/\Lambda_i' = \frac{1}{\tilde{r}_j} (\Lambda_{\ell(j-1)+i+1}/\Lambda_{\ell(j-1)+i})$.

  We will prove the following
  two statements.
  \begin{enumerate}
  \item[$(I)$] For any sublattice $\{\bm{0}\} \subset \tilde{\Lambda} \subseteq \Lambda'$ one has $\nd(\tilde{\Lambda}) \geq (2t)^{-1}$.
  \item[$(II)$] For any sublattice $\{\bm{0}\} \subset \tilde{\Lambda} \subseteq (\Lambda')^*$ one has $\nd(\tilde{\Lambda}) \geq (2t)^{-1}$.
  \end{enumerate}
  First we show $(I)$. We apply Lemma~\ref{lem:CanonicalFiltrationLowerBound} on $\Lambda'$ to obtain
  \[
   \nd(\tilde{\Lambda}) \ge t^{-1} \cdot \nd(\Lambda'_1) \ge t^{-1} \cdot \frac{r_{\ell(j-1)+1}}{\tilde{r}_j} \geq (2t)^{-1},
  \]
  since both numerator and denominator belong to the same interval $[2^\ell, 2 \cdot 2^\ell)$ for some $\ell \in \mathbb{Z}$.
  Next, we prove $(II)$. Given the filtration $\{\bm{0}\} = \Lambda'_0 \subset \cdots \subset \Lambda'_{k'} = \Lambda'$ with $U_i := \span(\Lambda'_i)$, the dual filtration is given by $\{\bm{0}\} = (\Lambda')^*_0 \subset \cdots \subset (\Lambda')^*_{k'} = (\Lambda')^*$ with $(\Lambda')^*_i := (\Lambda')^*\cap U_{k'-i}^\perp$ and determinant $\det((\Lambda')^*_i) = \det((\Lambda')^*) \cdot \det(\Lambda'_{k'-i}) = \det(\Lambda'_{k'-i})$, see for example~\cite{Dadush-Finding-DenseLatticeSubspacesSTOC19}. 
  Since quotients of the dual filtration are duals of the quotients of the original filtration, the dual filtration is also $t$-stable. We then apply Lemma~\ref{lem:CanonicalFiltrationLowerBound} on $(\Lambda')^*$:
 \[
   \nd(\tilde{\Lambda}) \ge t^{-1} \cdot \nd( (\Lambda')^*_1 ) = t^{-1} \cdot \Big(\frac{r_{\ell(j)}}{\tilde{r}_j}\Big)^{-1} \stackrel{r_{\ell(j)} \leq 2\cdot \tilde{r}_j}{\geq} (2t)^{-1}. \qedhere
  \]

\end{proof}

\section{Recent developments\label{sec:RecentDevelopments}}

Finally we would like to report on some recent exciting results that were announced after the preliminary version of this paper appeared.
\begin{definition}
  A convex body $K \subseteq \setR^n$ is in \emph{isotropic position} if (i) $b(K) = \bm{0}$ and (ii) $\E_{x \sim K}[xx^T] = I_n$.
\end{definition}
In other words, a convex body $K$ is isotropic if a uniform sample $x \sim K$ has the
same first and second moments as a standard Gaussian.
For any convex body $K \subseteq \setR^n$, there is an affine linear map $T : \setR^n \to \setR^n$
so that $T(K)$ is in isotropic position (see e.g. \cite[Chapter 10]{AsymptoticGeometricAnalysis-Book2015}).
  As we discussed previously in Lemma~\ref{lem:CovRadiusForStableLatticeWithBall}, Regev and Stephens-Davidowitz~\cite[Theorem 6.2]{Regev-SD-ReverseMinkowskiTheoremAnnalsOfMath2024} used the  Reverse Minkowski Theorem to prove that for any stable lattice $\Lambda \subseteq \setR^n$,  one has $\mu(\Lambda,B_2^n) \leq O(\sqrt{n} \log(2n))$. However, they also provided an upper bound
of $\mu(\Lambda,B_2^n) \leq O(\sqrt{n} \cdot L_n)$ with a different approach (see \cite[Theorem 6.7]{Regev-SD-ReverseMinkowskiTheoremAnnalsOfMath2024}), where
\[
 L_n := \sup\Big\{ \frac{1}{\sqrt{d}} \E_{x \sim K}\big[\|x\|_2^2\big]^{1/2} \mid K \subseteq \setR^d\textrm{ is an isotropic convex body and } d \leq n \Big\}
\]
is the \emph{isotropic constant}. 
Bourgain's \emph{Slicing Conjecture} stipulates that $L_n$ should indeed be upper bounded by a universal constant.
In a series of recent breakthroughs, Guan~\cite{guan2024notebourgainsslicingproblem} proved that $L_n \leq O(\log \log n)$, followed by
Klartag and Lehec~\cite{KlartagLehecProofOfSlicingConjectureGAFA2025} who finally show that $L_n \leq O(1)$.
We summarize the consequences for our setting:
\begin{lemma}[Covering radius for stable lattices III~\cite{Regev-SD-ReverseMinkowskiTheoremAnnalsOfMath2024,KlartagLehecProofOfSlicingConjectureGAFA2025}] \label{lem:CovRadiusForStableLatticeWithBallIII}
For any stable lattice $\Lambda \subseteq \setR^n$ one has  $\mu(\Lambda,B_2^n) \leq O(\sqrt{n})$.
\end{lemma}
Plugging in Lemma~\ref{lem:CovRadiusForStableLatticeWithBallIII} into the machinery of Regev and Stephens-Davidowitz~\cite{Regev-SD-ReverseMinkowskiTheoremAnnalsOfMath2024} also provides a tight upper bound on the covering radius of ellipsoids: 
\begin{theorem}[\cite{Regev-SD-ReverseMinkowskiTheoremSTOC17,Regev-SD-ReverseMinkowskiTheoremAnnalsOfMath2024}] \label{thm:KLForEllipsoidsIII} 
  For any full rank lattice $\Lambda \subseteq \setR^n$ and any ellipsoid $E \subseteq \setR^n$
  one has
  \[
   \mu_{KL}(\Lambda,\pazocal{E}) \leq \mu(\Lambda,\pazocal{E}) \leq O\Big(\sqrt{\log (2n)}\Big) \cdot \mu_{KL}(\Lambda,\pazocal{E}).
  \]
\end{theorem}
We note that \cite{Regev-SD-ReverseMinkowskiTheoremAnnalsOfMath2024} also found a lattice $\Lambda$
where $\mu(\Lambda,B_2^n) \geq \Omega(\sqrt{\log (2n)}) \cdot \mu_{KL}(\Lambda,B_2^n)$, hence the bound in Theorem~\ref{thm:KLForEllipsoidsIII} is best possible. 

On a related note, recall that we defined $C_{\textrm{cov}}(n)$ so that $\mu(\Lambda,Q) \leq C_{\textrm{cov}}(n) \cdot \ell_Q$
for 2-stable lattices $\Lambda$ and symmetric convex bodies $Q$ of dimension at most $n$ (see the end of Section~\ref{sec:FindingSubspace}).
 As far as we are aware, the best known bounds are still
 $1 \lesssim C_{\textrm{cov}}(n) \lesssim \log(2n)$ (see Prop~\ref{prop:CovRadiusOf2StableLattice}) and we leave tightening these bounds as a fascinating open problem. In particular any improvement on the upper bound would by Cor~\ref{cor:KLtypeInequWithSequenceOfSubspaces} immediately give an improvement on the $O(\log^3(2n))$ factor from
 Theorem~\ref{thm:KLConj}.
 At least for the integer lattice $\mathbb{Z}^n$, one can show that for any symmetric convex body $Q \subseteq \mathbb{R}^n$ and any $x \in \mathbb{R}^n$, the random vector $z \in \mathbb{Z}^n$ with independent coordinates so that $\E[z] = x$ satisfies $\E[ \|x-z\|_Q] \le \sqrt{2\pi} \ell_Q$, so that indeed $\mu(\mathbb{Z}^n, Q) \lesssim \ell_Q$.  

A crucial ingredient for our proof has been the $\ell$-position for symmetric convex bodies from Theorem~\ref{thm:PisierRescaling}. So far we have only used the symbol $\| \cdot \|_Q$ when $Q$ is a symmetric convex body. But one can extend this notation to  $\|x\|_K := \min\{ r \geq 0 \mid x \in rK\}$ for  any convex body $K \subseteq \setR^n$ with $\bm{0}$ in its interior.
Then $\| \cdot \|_K$ is not a norm anymore, but a \emph{gauge function} (or \emph{Minkowski functional}). Similarly we can extend the $\ell$-value $\ell_K := \E_{x \sim N(\bm{0},I_n)}[\|x\|_K^2]^{1/2}$. Unfortunately, the proof of Theorem~\ref{thm:PisierRescaling} itself does inherently rely on symmetry. It had been conjectured that a similar statement
should still be true for the asymmetric case. This has very recently been resolved in another breakthrough:
\begin{theorem}[Bizeul, Klartag~{\cite[Cor 1.4]{bizeul2025distancesnonsymmetricconvexbodies}}] \label{thm:LValueForIsotropicPosition}
Let $K \subseteq \setR^n$ be a convex body in isotropic position. Then
$
  \ell_{K} \cdot \ell_{K^{\circ}} \leq O(n \log^3(2n))
$.
\end{theorem}
Using Theorem~\ref{thm:LValueForIsotropicPosition} one could reprove our main result (Theorem~\ref{thm:KLConj}) with a weaker inequality of
$\mu(\Lambda,K) \leq O(\log^5 (2n)) \cdot \mu_{KL}(\Lambda,K)$ without having to go via the symmetrizer of $K$.
Additionally, as noted in \cite{bizeul2025distancesnonsymmetricconvexbodies}, combining Theorem~\ref{thm:LValueForIsotropicPosition} with the frameworks of
Banaszczyk~\cite{Banaszczyk1996TransferenceTheoremsForGeneralConvexBodies} and
Banaszczyk, Litvak, Pajor and Szarek~\cite{BanaszczykLitvakPajorSzarekMOR99Flatness}
one can obtain a bound of $O(n \log^3(2n))$ on the flatness
constant without having to rely on our Theorem~\ref{thm:KLConj}.

We would also like to mention that building upon the work of~\cite{bizeul2025distancesnonsymmetricconvexbodies}, there has been an improvement on the exponent of $n/d$ in Proposition~\ref{prop:VolumeOfProjectionVsSymmetrizer}:

\begin{proposition}[{\cite[Theorem 1.1]{tzioutziou2025}}] \label{prop:VolumeOfProjectionVsSymmetrizerImproved}
Let $K \subseteq \setR^n$ be a convex body so that $b(K) = \bm{0}$ or $b(K^\circ) = \bm{0}$ and let $F \subseteq \setR^n$ be a $d$-dimensional subspace. 
Then 
\[
 \Vol_d(\Pi_F(K))^{1/d} \lesssim \frac{n}{d} \cdot \log(2n)^3  \cdot \Vol_d(\Pi_F(K_{\textrm{sym}}))^{1/d}.
\]
\end{proposition}

And finally there is the motivating question for this paper --- what is the optimum running time to solve
arbitrary $n$-variable integer linear programs $\max\{ c^Tx \mid Ax \leq b, x \in \setZ^n\}$? A conjecture
popularized by Friedrich Eisenbrand is that time $2^{O(n)}$ times a polynomial in the encoding length of $A$, $b$ and $c$ should suffice, though no candidate approach is known to the authors. For example, consider a
stable lattice $\Lambda \subseteq \setR^n$ and let $Q \subseteq \setR^n$ be a symmetric convex body in $\ell$-position scaled so that $\mu(\Lambda,Q) = 1$. Then for any $u \in \setR^n$, one has $|\Lambda \cap (u+Q)| \leq (\log(2n))^{O(n)}$ and Dadush's algorithm can find a lattice point in $\Lambda \cap (u+Q)$ in time $(\log(2n))^{O(n)}$
without any recursion, just by covering $u+Q$ with ellipsoids. But even for this case we do not know
how to improve beyond the time of $(\log(2n))^{O(n)}$.

\paragraph{Acknowledgements.} The authors are grateful to Daniel Dadush for numerous discussions on
related topics, a careful reading of a preliminary draft, and for the proof of the improved bound in Theorem~\ref{thm:FlatnessConstant}. The authors would also like to thank the anonymous reviewers who made numerous helpful suggestions.

\bibliographystyle{alphaurl} 
\bibliography{KLconjecture}

\end{document}